\documentclass[10pt]{amsart}
\usepackage{amsmath,amsfonts,amssymb,amsthm,amscd,latexsym,euscript,xypic, verbatim}

\usepackage{epsf}
\usepackage{epsfig}
\usepackage{amstext}
\usepackage{xspace}
\usepackage{amsfonts}
\usepackage{amsmath}
\usepackage{amssymb}
\usepackage{xspace}
\usepackage{graphicx}
\usepackage[active]{srcltx}
\usepackage[all]{xy}
\usepackage{graphics}
\usepackage{color}
\usepackage{tikz}
\usepackage{caption}
\usepackage[hidelinks]{hyperref}
\usepackage{multirow}
\usetikzlibrary{patterns}
\usepackage{fp}

\colorlet{agreen}{green!50!black}
\swapnumbers
\theoremstyle{plain}

\newtheorem*{T1}{Theorem 1}
\newtheorem*{T2}{Theorem 2}

\newtheorem*{thm1.2}{(1.2) Theorem}
\newtheorem*{thm1.3}{(1.3) Theorem}
\newtheorem*{thm1.4}{(1.4) Theorem}
\newtheorem*{propA*}{Proposition A}
\newtheorem*{propB*}{Proposition B}
\newtheorem*{thmC*}{Theorem C}
\newtheorem*{propD*}{Proposition D}
\newtheorem{prop}{Proposition}[section]
\newtheorem{thm}[prop]{Theorem}

\newtheorem{cor}[prop]{Corollary}

\theoremstyle{definition}
\newtheorem{point}[prop]{}

\newtheorem{Def}[prop]{Definition}

\newtheorem*{Def*}{Definition}

\newtheorem{example}[prop]{Example}

\newtheorem{notation}[prop]{Notation}
\newtheorem*{notation*}{Notation}
\newtheorem*{question*}{Question}
\newtheorem{remark}[prop]{Remark}

\newcommand{\Q}{{ \mathbf  Q}_+}
\newcommand{\VVect}{{\text{Vect}_K}}

\newcommand{\D}{\mathcal D}

\newcommand{\Tame}{\mathbb{T}(\Q^r, \VVect)}

\title{Stable Invariants for Multiparameter Persistence}

\author[O. G\" afvert]{Oliver G\" afvert}
\author[W. Chach\'olski]{Wojciech Chach\'olski}

\begin{document}

\begin{abstract}
In this paper we explain how to convert discrete invariants into stable ones via what we call hierarchical stabilization. We illustrate  this process by constructing stable invariants for multi-parameter persistence modules with respect to the interleaving distance and so called simple noise systems. For one parameter, we recover
the standard barcode information. For more than one parameter we prove that  the constructed invariants are in general NP-hard to calculate. A consequence is that computing the feature counting function, proposed by Scolamiero et. al. (2016), is NP-hard.
\end{abstract}

\maketitle

\section{Introduction}

Understanding the underlying structure generating observed data is the aim of data analysis. A dataset is typically given in the form of a finite set $\mathcal U$ together with a sequence of $r$ maps out of $ \mathcal U$ into metric spaces called measurements. One way to analyze such data  is to use these  measurements to  produce summaries or signatures describing the structure of the data with a hope of gaining some insight into the source.
To reduce the dependency of the signatures on the accuracy of the measurements, which is necessary in the presence  of  heterogeneity, noise, variability, missing information etc., one can   utilize homological tools from  algebraic topology. Constructing and studying  homology based stable summaries is the aim of topological data analysis (TDA). For $r=1$ (only one measurement), this 
approach  has been applied very successfully  to, for example, data coming from  medicine \cite{Adcock201436, Nicolau26042011, HBM:HBM22521}, biology \cite{Chan12112013, 10.1371/journal.pcbi.1002581}, robotics \cite{7078886, pokorny2016b} and image processing \cite{bendich2010computing, doi:10.1137/070711669, 2011arXiv1112.1993A}. 

For $r>1$,  mathematical foundations still need to be developed   to construct  and understand  homology based stable signatures. This  is important  as studying correlations between different measurements is essential in data analysis. The main object of study in multi-parameter persistent homology is the \textit{multi-parameter persistence module}. This module is constructed based on the $r$ measurements on the dataset in order to capture their multi-variate relationships. The potential and value of the multi-parameter topological data analysis 
 can be seen for example in  \cite{Adcock201436} where information extracted from a $2$-parameter rank invariant yields a better classification rate than the  $1$-parameter signatures given by  barcodes.  In section~\ref{sec:hier}
 we present a simple procedure  inspired by hierarchical clustering which in our  mind is the essence of persistence. This procedure, which we call hierarchical stabilization,  is about  converting  discrete invariants into stable ones. We then use stable invariants to define continuous  multi-parameter   invariants such as  stabilized Betti numbers. The hierarchical stabilization depends on a choice of pseudometric, and the pseudometrics we consider in this paper are the ones arising from \textit{noise systems}, defined in Scolamiero et. al. in \cite{martina}. For this class of pseudometrics we can perform computations, and we show how to compute the hierarchical stabilization of the 0-th Betti number (Corollary \ref{cor redfinite}). This class is also shown to contain elements which are equivalent, in the metric sense, to the interleaving distance of \cite{MR3348168} (Proposition \ref{prop:metriceq}).
 
 We isolate a subclass of noise systems which we call \textit{simple noise systems}, containing the \textit{standard noise system} of \cite{martina}. Noise systems in this subclass can be defined constructively using what we call \textit{persistence contours} (Theorem 9.6), and are thus amendable for performing computations. The persistence contours are an essential tool in proving our main result (Theorem~\ref{thm main}) and for showing the correspondence with the interleaving distance (Theorem 12.4).
 
 Besides the lack of invariants, computational challenges are further road blocks for  making    multi-parameter  homological stable signatures  applicable. In this paper we define an algorithm to compute the stabilized 0-th Betti number for simple noise systems and in Corollary \ref{cor redfinite} we show that this computation requires only a finite number of arithmetic operations when computing with coefficients in a finite field. Our main result, however, is proving that computing  the stabilization of the 0-th Betti number is in general NP-hard when $r\geq 2$.
   \begin{T1}\label{thm main}
  	Computing the stabilization of the 0-th Betti number is NP-hard.
  \end{T1} 
The stabilized 0-th Betti number is equivalent to \textit{feature counting function}, introduced by Scolamiero at. al.  \cite{martina} and our result thus proves that computing this invariant is in general NP-hard. 

Building on our main result, we show that the hardness result still holds when computing the stabilized 0-th Betti number with respect to the interleaving distance.
  \begin{T2}
 	Computing the stabilization of the 0-th Betti number with respect to the interleaving distance is NP-hard.
 \end{T2}
 
 Since for $r=1$, the stabilization of the 0-th Betti number is similar to the usual barcode (Proposition~\ref{prop recovstandard}), our main  result  illustrates  computational difference between  $1$ and multi-parameter situations and why understanding  correlations is a much harder problem.   

For $r>1$ the moduli of multi-parameter persistence modules is a very complicated algebraic variety (\cite{Carlsson2009}), it is not possible    to classify    all such modules by easily visualizable, computable, and  stable invariants.  For $r>1$, instead of  a complete classification (as it is for $r=1$), we need other methods of  defining and extracting 
 stable  information about the data system out of the associated multi-parameter persistence module. It is exactly for that purpose we have developed the hierarchical stabilization. The need for a stabilization procedure comes from the fact that algebraic invariants of multi-parameter persistence modules such as minimal number of generators, Betti tables, Hilbert polynomials etc.\  tend to change drastically when  the initial data is altered even slightly.  That is why the classical commutative algebra invariants such as those investigated by Harrington et. al. in \cite{hal} for the multi-parameter persistence setting, need to be stabilized in order to be useful for data analysis.
 
  The most prominent stable invariant for multi-parameter persistence is the rank invariant, proposed by Carlsson and Zomorodian \cite{Carlsson2009}. This invariant can be seen as collecting all one-dimensional barcode information in a big table. Consequently, it captures the persistent information of the module, but extracting \textit{multivariate} persistent information from the table and visualizing it require more work. On the other hand, the feature counting function \cite{martina} extracts readably persistent information of the module. In this paper we generalize its construction through the process of hierarchical stabilization and our main result (Theorem \ref{thm main}) show that computing it is in general NP-hard. This is an indication that extracting multivariate persistent information from the rank invariant, and perhaps in general, might be NP-hard as well.
  
  The entire paper is written with implementation in mind.  In the first part of the paper, up to Section \ref{sec:simplenoise}, we recall the framework of \cite{martina} from a computational perspective and use it to define our stabilization procedure. For all the introduced or recalled  concepts, such us Betti diagrams and numbers, considered  noise systems,   related shifts, and the stabilized 0-th Betti number, we indicate how they can be calculated or implemented. The latter part of the paper is focused on defining algorithms to compute the constructed invariants and showing their complexity of computation.  
  
  
 
 \subsection*{Organization}The paper is organized as follows: Section 2 introduces notation and background on topological data analysis. Section 3 introduces the notion of hierarchical stabilization, which is a generalization of the invariant proposed in \cite{martina}. In Section 4-6 we review the foundations of the framework introduced in \cite{martina} from a computational perspective and describe how hierarchical stabilization can be used to stabilize classical invariants of algebraic topology. We indicate how considered invariants, noise systems, and related shits can be calculated and implemented. In Section 7 we describe the notion of noise systems introduced in \cite{martina} and in Section 8 we define a smaller class of noise systems called \textit{simple noise systems}. In Section 9 we introduce persistence contours and show (Theorem 9.6) that they provide a constructive definition of simple noise systems. In Section 10 we show how the stabilized 0-th Betti number is translated in the 1-parameter case and that by considering \textit{truncations}, we can recover the barcode (Proposition 10.3). In Section 11 we show that computing the stabilized 0-th Betti number with respect to simple noise systems is in general NP-hard (Theorem 11.6) and in Section 12 we extend this result to hold for the interleaving distance (Theorem 12.4).




\section{Notation \& background}
\begin{point}\label{pt NQ}
We use  bold letters and   black-board bold letters  to denote small categories and  their sets of objects respectively. For example $\mathbf I$ denotes the small category whose set of objects is denoted by ${\mathbb I}$.

Let $\mathbf I$ be a poset, $S\subset {\mathbb I}$  a subset of its objects, and $i$ an object in $\mathbf I$.
We write $S\leq i$ if $j\leq i$ for any $j$ in $S$.

For a poset $\mathbf I$, we use the symbol ${\mathbf I}_{\infty}$ to denote the poset  obtained 
from  $\mathbf I$ by adding an extra object $\infty$ such that $i\leq\infty$ for any  $i$  in ${ \mathbb   I}$.

For a field $K$, the symbols $\text{Vect}_K$ denotes the category of $K$-vector spaces.
Let $F\colon {\mathbf I}\to \text{Vect}_K$ be a functor. For an element  $g\in F(i)$, the object $i$ is called the
{\bf coordinate} of $g$.

Let ${ \mathbf  N}$, ${ \mathbf  Q}$, and  ${ \mathbf  R}$ denote the {\bf posets} of natural numbers $\{0<1<\ldots\}$,  {\em non-negative} rational and real numbers respectively. Thus  ${ \mathbf  N}$, ${ \mathbf  Q}$, and  ${ \mathbf  R}$ are categories with the set of morphisms from  $v$ to $w$ being empty if $v\not\leq w$ and having a unique morphism $v\leq w$ otherwise.  According to our convention, the symbols ${ \mathbb   N}$, ${ \mathbb   Q}$, ${ \mathbb   R}$ denote  the {\bf sets} of 
natural,  {\em non-negative} rational and real numbers respectively. Similarly, ${ \mathbb   N}^r$ and ${ \mathbb   Q}^r$ denote the sets of $r$-tuples of natural and non-negative rational numbers, while  ${ \mathbf  N}^r$ and ${ \mathbf  Q}^r$ denote the posets of $r$-tuples of natural and {\em non-negative} rational numbers respectively, with $(v_1,\ldots,v_r)\leq (w_1,\ldots,w_r)$   if   $v_i\leq w_i$, for all $i$.
 According to this notation, $f\colon { \mathbb   Q}^r\to { \mathbb   N}^{r}$ denotes a function between sets while
 $g\colon { \mathbf  Q}^r\to { \mathbf  N}^r$ denotes a functor between categories, which in this case is a function with the property 
  $g(v)\leq g(w)$ for any $v\leq w$ in ${ \mathbf  Q}^r$.

 The symbol $e_n$  denotes the element in  ${ \mathbb   N}^r$ whose  coordinates are $0$ except the $n$-th coordinate which is $1$. For a subset $S\subset \{1,\ldots, r\}$, we set $e_S:=\sum_{n\in S}e_n$.  Elements of  
 $S\subset \{1,\ldots, r\}$  are naturally ordered by their size $S=\{n_1<\cdots <n_{|S|}\}$ and we  call the index $i$ of the element $n_i$ its {\bf order}  in $S$. 
   \end{point}
\begin{point}\label{pt multsubset}
Let $T$ be a set. A {\bf multi-subset} of $T$ is by definition a function   $f\colon T\to{ \mathbb  N}$.  
 For such a  multi-subset, its value $f(t)$  is called the multiplicity of $t$. The set $\{t\in  T\ |\ f(t)\not=0\}$ is called the {\bf support} of $f$ and is denoted by $\text{supp}(f)$.  A multi-subset is called {\bf finite} if its support is finite, in which case the sum $\sum_{t\in  T}f(t)$ is  called the {\bf rank} of $f$ and denoted by $\text{rank}(f)$. 
We use  symbols
$\text{Mult}(T)$  to denote  the set of all multi-subsets  of $T$.
\end{point}
\begin{point}\label{pt pseudometric}
An extended pseudometric on a set $T$ is a function $d\colon T\times T\to { \mathbb   R} \cup \{\infty\}$ such that : $d(x, x) = 0$, $d(x, y) = d(y, x)$, and  $d(x, z) \leq d(x, y) + d(y, z)$ for any $x$,$y$, and $z$ in $T$ (here  we of course take $r\leq \infty$ and $\infty+r=\infty$ for any $r$ in  $ { \mathbb   R}\cup \{\infty\}$). 
\end{point}

\subsection{Topological Data Analysis}The first step in TDA is to transform a data system (a finite set $ \mathcal U$ with  $r$ measurements on it), via
for example the \u{C}ech, Vietoris-Rips, or  their  witness version constructions, 
into a  multi-parameter simplicial complex. The result of this step is a functor  $F\colon { \mathbf  Q}^r\to \text{Spaces}$
indexed by the poset  ${ \mathbf  Q}^r$ of $r$ tuples of non-negative rational numbers where $(v_1,\ldots ,v_r)\geq (w_1,\ldots ,w_r)$ if and only if
$v_i\geq w_i$ for all $i$. By applying  homology with coefficients in a field  $K$ we get a so called multi-parameter persistence module $H_i(F, K)\colon  { \mathbf  Q}^r\to\text{Vect}_K$ (a functor indexed by ${ \mathbf  Q}^r$ with values in the category of $K$ vector spaces).
This step is important as  it relaxes the dependency on the accuracy of the measurements. Since $ \mathcal U$ is a finite set, the obtained functors are quite special: they are   finitely generated  and tame~\cite{martina}. The category  of tame functors with values in a category of vector spaces has properties similar to the category of  graded modules over the polynomial ring in $r$ variables (\cite{martina}).  In particular all tame functors have free resolutions
of length no grater than $r$. It follows that for $r=1$, similarly to  finitely generated modules over a PID,   finitely generated and tame functors
can   be classified. The barcode (see Section \ref{sec:r1}) is a  particularly useful  form of this   classification for data analysis purposes since barcodes are not only efficiently computable in terms of the  data system,  but more importantly, barcodes are also stable    (small changes in the data lead to small changes in its barcode) \cite{barcodes}. For $r>1$ it was shown by Carlsson et. al. \cite{Carlsson2009} that no complete invariant exists. Therefore it is not possible    to classify    all multi-parameter persistence modules by easily computable, and  stable invariants.  For $r>1$  we therefore need other methods of  defining and extracting 
stable  information about the data system out of the associated multi-parameter persistence module. In the next section we detail how unstable invariants can be stabilized through the process of hierarchical stabilization. 


\section{Hierarchical  stabilization}\label{sec:hier}
The aim of clustering is  to partition  a data set into parts that aggregate elements  sharing similar characteristics
(whatever we choose them to be) and separate elements with different characteristics. One way to make a decision about which partition to choose  is to assemble possible partitions into a dendrogram and study how different partitions  are related to each other. Producing dendrograms, not just partitions, is what a hierarchical clustering is about. Dendrograms have an important advantage over  partitions. They form a metric space (see for example~\cite{MR2645457}), which means that one  can measure how close or how far apart dendrograms can be. In this way one can study  stability of  a particular hierarchical clustering method, which is essential in data analysis. In this section we use the same hierarchical idea to show how to turn discrete invariants  into    stable and continuous  ones.

  Consider the set $\text{Mult}(  { \mathbb   Q})$ of multi-subsets of ${ \mathbb   Q}$ (see~\ref{pt multsubset})
which are simply functions of the form $f\colon { \mathbb   Q}\to { \mathbb   N}$.
For  $\epsilon$ in  $ { \mathbb   Q}$,  two multi-sets
$f,g\colon   { \mathbb   Q}\to  { \mathbb   N}$ are defined to be {\bf $\epsilon$-close} if,
for any $\tau$ in $  { \mathbb   Q}$, the following inequalities hold: $g(\tau)\geq f(\tau+\epsilon)$ and $f(\tau)\geq g(\tau+\epsilon)$. This leads to  an extended  pseudometric (see~\ref{pt pseudometric}) on the set  $\text{Mult}(  { \mathbb   Q})$  called the {\bf interleaving distance}:
\[d(f,g):=\begin{cases}
\infty & \text{if $f$ and $g$ are not  $\epsilon$-close for any $\epsilon$}\\
\text{inf}\{\epsilon\ |\ \text{$f$ and $g$ are $\epsilon$-close}\} &
 \text{otherwise}
\end{cases}\]
From now on the symbol 
$\text{Mult}(  { \mathbb   Q})$  is used to denote the metric space of multi-subsets of ${ \mathbb   Q}$ with the interleaving distance as a metric. In the context of persistence, interleaving distances have been introduced in~\cite{persistanseChazal}  and have since then  been  extensively studied (see for example~\cite{MR3348168}).

Let $T$ be a set.  A discrete invariant on $T$ is simply a function $f\colon T\to  { \mathbb   N}$. The aim of this section  is to explain how to stabilize $f$ to obtain a continuous  invariant.  
The first step is   to choose  an extended pseudometric $d\colon T\times T\to  { \mathbb   R}\cup \{\infty\}$ on $T$ (we can not talk about stability and continuity without being able to measure distances).
For any such  choice, we construct a function $\hat{f}\colon T\to \text{Mult}(  { \mathbb   Q})$
called  the  {\bf (hierarchical) stabilization of $f$}. For   $x$ in $T$ and $\tau$  in $ { \mathbb   Q}$ define:
 \[\hat{f}(x)(\tau):=\text{min}\{f(y)\mid d(x,y)\leq \tau\}\]
Thus the value of the multi-set $\hat{f}(x)$ at $\tau$ is the smallest value   $f$  takes on the disc
 $B(x,\tau):=\{y\in T\ |\ d(x,y)\leq \tau\}$. We stress again that the stabilization $\hat{f}$ of $f$ 
 depends on the choice of an extended   pseudometric on $T$. Note that if  $\epsilon\geq \tau$, then $B(x,\tau)\subset B(x,\epsilon)$ and consequently 
$ \hat{f}(x)(\tau) \geq  \hat{f}(x)(\epsilon)$. 

 \begin{prop}\label{prop:stable} 
 For any choice of an extended pseudometric on $T$,
the function  $\hat{f}\colon T\to \text{\rm Mult}(  { \mathbb   Q})$  is $1$-Lipschitz, i.e. for any $x$ and $y$ in $T$,
$d(x,y)\geq d(\hat{f}(x),\hat{f}(y))$.
 \end{prop}
 \begin{proof}
 If $d(x,y)=\infty$ there is nothing to prove. Assume $d(x,y)<\infty$. 
 By the triangle inequality, for any $\tau$ and $\epsilon$ in $ { \mathbb   Q}$ such that $\epsilon\geq d(x,y)$, we  have inclusions $B(y,\tau)\subset B(x,\tau+\epsilon)$ and
  $B(x,\tau)\subset B(y,\tau+\epsilon)$. It then follows that $\hat{f}(y)(\tau)\geq \hat{f}(x)({\tau+\epsilon})$ and
   $\hat{f}(x)({\tau})\geq \hat{f}(y)({\tau+\epsilon})$. That means that $\hat{f}(x)$ and $\hat{f}(y)$ are $\epsilon$-close. As this happens for all
   $\epsilon\geq d(x,y)$, we can conclude $d(x,y)\geq d(\hat{f}(x),\hat{f}(y))$.
    \end{proof}

The input for the  hierarchical  stabilization has three ingredients: (i) a set $T$, (ii) a discrete invariant $f\colon T\to { \mathbb   N}$, and  (iii) a choice of an extended pseudometric on $T$. The outcome is a $1$-Lipschitz function $\hat{f}\colon T\to \text{Mult}({ \mathbb   Q})$. In the next few sections we      illustrate how 
to apply  this hierarchical  stabilization  process  when (i) $T$ is the set of   finitely generated tame functors (see~\ref{pt bettieuler}),
(ii) $f\colon T\to { \mathbb   N}$ is a classical homological invariant such as the $i$-th Betti number (see~\ref{pt bettieuler}),  and
(iii) the  pseudometric on $T$ is induced by a noise system (see Section~\ref{sec noise}). 
Our aim for this article has been to
 show that, for $r\geq 2$, calculating the stabilization of  the $0$-th Betti number
is in general an NP-hard problem, which we prove in Section~\ref{sec NPhardness}. 

\section{Homological invariants of vector space valued  functors}
Let $K$ be a field and ${\mathbf I}$   a poset. In this section we recall how  homological invariants of certain functors of the form $F\colon{\mathbf I}\to  \text{Vect}_K$  are defined and constructed.

\begin{point}\label{pt free}
{\bf Freeness.}
For   $i$ in ${\mathbb I}$,
$K_{\mathbf I}(i,-)\colon {\mathbf I} \to \text{Vect}_K$ denotes the composition of the representable functor $\text{mor}_{\mathbf I}(i,-)\colon {\mathbf I} \to \text{Sets}$ with 
the linear span functor $K \colon  \text{Sets}\to \text{Vect}_K$. We often omit the subscript ${\mathbf I}$ and write $K(i,-)$. Since $\mathbf I$ is a poset, $K(i,j)=K$ if $i\leq j$ and $K(i,j)=0$ if $i\not\leq j$. The functor $K(i,-)$ has the following Yoneda property explaining why it is called {\bf free on one generator in degree $i$}. For any $F\colon {\mathbf I}\to   \text{Vect}_K$ the homomorphism $\text{Nat}(K(i,-), F)\to F(i)$ that assigns to a natural transformation 
$\phi\colon K(i,-)\to F$ the element $\phi_i(\text{id}_i)$ in $F(i)$ is an isomorphism. In this way  any element 
 $g$ in $F(i)$  (an element with coordinate $i$, see~\ref{pt NQ}) determines a unique natural transformation, denoted by the same symbol $g\colon K(i,-)\to F$, which  maps the element $\text{id}_i$ in $K(i,i)$ to $g$ in $F(i)$. In particular there is a non-zero natural transformation $K(i,-)\to K(j,-)$ if and only if $j\leq i$. Moreover, any such non-zero natural transformation  is a monomorphism. It follows that  $K(i,-)$ and $K(j,-)$ are isomorphic if and only if $i=j$.

Functors of the form 
$\oplus_{i\in {\mathbb I}}( K(i,-)\otimes V_i)$ are called {\bf free}. 
By the  Yoneda property, for any $F\colon{\mathbf I}\to\text{Vect}_K$, a sequence of elements $\{g_s\in F(i_s)\}_{s\in S}$  determines a unique
natural transformation denoted by $[g_s]_{s\in S}\colon \oplus_{s\in S}K(i_s,-)\to F$.
This Yoneda property and 
 the fact that $\mathbf I$ is a poset  imply
that two free functors $\oplus_{i\in {\mathbb I}} (K(i,-)\otimes V_i)$ and $\oplus_{i\in {\mathbb I}}( K(i,-)\otimes W_i)$  are isomorphic if and only if, for any $i$ in  ${\mathbb I}$, the vector spaces $V_i$ and $W_i$ are isomorphic.
Thus a free functor  is up to an isomorphism determined by the sequence of vector spaces $\{V_i\}_{i\in{\mathbb I}}$.
Based on the properties of this sequence, we  use the following dictionary about a free functor $P=\oplus_{i\in {\mathbb I}}( K(i,-)\otimes V_i)$:
\begin{enumerate}
\item The vector space $V_i$ is called the {\bf $0$-th homology} of $P$ at $i$ and is denoted by $H_0P(i)$. With this notation $P$ is 
isomorphic to $\oplus_{i\in {\mathbb I}} (K(i,-)\otimes H_0P(i))$.
\item  The set $\{i\in {\mathbb I}\ |\ H_0P(i)\not=0\}$ is called the {\bf support} of $P$ and is denoted by $\text{supp}(P)$.
\item $P$ is  of {\bf finite type} if $H_0P(i)$ is finite-dimensional for all $i$.
\item If  $P$ is of  finite type,  the multi-subset 
$\beta_0P(-)\colon {\mathbb I}\to {\mathbb N}$ of ${\mathbb I}$, defined as $ \beta_0P(i):= \text{dim}\, H_0P(i)$,
is called the {\bf Betti diagram} of $P$. Two finite type free functors are isomorphic if and only if they have  the same Betti diagrams.
\item  $P$ is   of {\bf finite rank} if it is of finite type and its Betti diagram $ \beta_0P(-)$ is a  finite multi-set
(see~\ref{pt multsubset}). \item  Let $P$ be of finite rank. The number   $\sum_{i\in{\mathbb I}}\text{dim}\,H_0P(i)$ is also called the {\bf rank} or the {\bf Betti number} of $P$ and is denoted  by $ \beta_0P$.
\end{enumerate}
\end{point}
\begin{point}\label{pt freecover}
{\bf Minimality.}
Recall that a morphism $\phi\colon X\to Y$ in a category is called {\bf minimal} if any morphism  $f\colon X\to X$ satisfying   $\phi=\phi f$ is an isomorphism (see~\cite{MR1476671}). A natural transformation $\phi\colon P\to F$   of functors indexed by $\mathbf I$ with values in $\text{Vect}_K$
is called a {\bf minimal cover}  of $F$ if $P$  is free and $\phi$ is both minimal and an epimorphism.
Minimal covers 
are unique up to an isomorphism: if  $\phi\colon P\to F$ and  $\phi'\colon P'\to F$
are minimal covers  of $F$, then there is an isomorphism (not necessarily unique) $f\colon P\to P'$ such that
$\phi=\phi' f$. Furthermore any  $g\colon P\to P'$,  for which $\phi=\phi' g$, is an isomorphism  (a consequence of minimality). Note however that in this generality a minimal cover may not exist.

A set  $\{g_s\in F(i_s)\}_{s\in S}$   {\bf generates} a functor
  $F\colon { \mathbf  I}\to \text{Vect}_K$  if the  induced natural transformation
$[g_s]_{s\in S}\colon \oplus_{s\in S}K(i_s,-)\to F$  
 is an epimorphism. 
A functor   is  {\bf finitely generated} if it is generated by a finite set. It is called {\bf cyclic} if it is generated by one element.
A  free functor is finitely generated if and only if it is of finite rank.
For a finitely generated   $F$, a set 
 $\{g_s\in F(i_s)\}_{s=1}^{n}$ is called a {\bf minimal  set of generators}  if it  generates $F$ and 
 no sequence with fewer elements can generate $F$.
 
 Assume  $F$ is finitely generated and admits a minimal cover
 $P\to F$.   Then $P$ is free (by definition) and  also finitely generated.
For such a functor admitting a minimal cover,   $\{g_s\in F(i_s)\}_{1\leq s  \leq n}$  is its minimal set of generators 
if and only if the  natural transformation
$ [g_s]_{1\leq s\leq n }\colon \oplus_{s= 1}^{n}K(i_s,-)\to F$
 is a minimal cover. It follows that all minimal sets of generators of $F$ have the following common property:
  the number of generators in such  a set that belong to $F(i)$ is equal to  the value of the Betti diagram $\beta_0P(i)$. In particular, the number of elements in 
 a minimal set of generators of $F$  is equal to   the rank of $P$ and  the set of coordinates
 of elements in a minimal set of generators of $F$ coincides with  $\text{supp}(P)$.
 \end{point}

\begin{point}\label{pt bettigeneral}
{\bf Resolutions.}
An exact sequence $\cdots\to P_n\xrightarrow{\delta_n}\cdots\xrightarrow{\delta_2} P_1\xrightarrow{\delta_1}P_0\xrightarrow{\delta_0} F\xrightarrow{\delta_{-1}} 0$ of functors indexed by $\mathbf I$ with values in $\text{Vect}_K$   is called a {\bf minimal free  resolution of $F$}  if 
 $P_n\to \text{Ker}(\delta_{n-1})$ is a minimal cover for any $n\geq 0$ (see~\ref{pt freecover}). Any two minimal resolutions of $F$ are isomorphic.
  Assume  $F\colon {\mathbf I}\to  \text{Vect}_K$ admits a minimal resolution as above. For such a functor we are going to use the following dictionary: 
 \begin{enumerate}
 \item The vector space $H_0P_n(i)$ (see~\ref{pt free}) is called the {\bf $n$-th homology} of $F$ at $i$ and is denoted by $H_nF(i)$. With this notation, $P_n$ is isomorphic to $\oplus_{i\in {\mathbb I}}(K(i,-)\otimes H_nF(i))$.
 \item The  set $\{i\in{\mathbb I}\ |\ H_0F(i)\not=0\}$ is called the {\bf support} of $F$ and is  denoted  by $\text{supp}(F)$.
 \item $F$ is   of   {\bf finite type} if  $P_n$ is of finite type for any $n\geq 0$ (see~\ref{pt free}).
 \item If $F$ is of finite type, the multi-subset $\beta_n F(-)\colon{\mathbb I}\to {\mathbb N}$ of ${\mathbb I}$, defined as $\beta_n F(i):=\text{dim}\,H_nF(i)$, is called the {\bf $n$-th Betti diagram} of $F$.
  \item $F$ is  of   {\bf finite rank} if  $P_n$ is of finite rank for any $n\geq 0$ (see~\ref{pt free}).
Thus  $F$ is   of   finite rank if and only if   it is of finite type and all its Betti diagrams $\beta_n F$ are finite for all $n\geq 0$.
 \item Let $F$  be of finite rank. For $n\geq 0$, the number  $\sum_{i\in{\mathbb I}} \beta_n F(i)=\sum_{i\in{\mathbb I}}\text{dim}\,H_nF(i)$ is called the {\bf $n$-th Betti number} of $F$ and is denoted by $\beta_n F$.  We also use the term {\bf  rank} of $F$  to  denote its $0$-th Betti number.
 \end{enumerate}
 Note that all the notions defined in this paragraph do not depend on the choice of a minimal free resolution of $F$. 
  \end{point}
  
 \begin{point} \label{pt bar}
  {\bf Bars.}  Let $a\leq b$ be  comparable objects in  the poset $\mathbf I$. Consider the natural transformation $K(b,-)\to  K(a,-)$ induced by the element $a\leq b$   in   $\text{mor}_{\mathbf I}(a,b)$. The cokernel of this natural transformation  is denoted by $[a,b)$. The functor $[a,b)$ is called the {\bf bar} starting in $a$ and ending in $b$. The exact sequence $0\to K(b,-)\to  K(a,-)\to [a,b)\to 0$ is a minimal  free resolution.  Consequently:
 \[H_n[a,b)(i)=\begin{cases}
 K&\text{if } n=0, i=a\\
 K & \text{if } n=1, i=b\\
 0&\text{otherwise}
 \end{cases}\ \ \ \ \ 
 \beta_n[a,b)(i)=\begin{cases}
 1&\text{if } n=0, i=a\\
 1 & \text{if } n=1, i=b\\
 0&\text{otherwise}
 \end{cases}
 \]
\[\beta_n[a,b)=\begin{cases}
 1&\text{if } n=0\text{ or } n=1\\
 0&\text{otherwise}
 \end{cases}\]

  \end{point}

\section{Functors indexed by ${ \mathbf  N}^r$}
The aim of this section is to recall how to effectively calculate the homology, the Betti diagrams, and Betti numbers of
functors of the form 
   $F\colon { \mathbf  N}^r\to \text{Vect}_K$. We  call such functors  {\bf frames}. All the material presented here  is standard  as the category of  frames  is equivalent to the category of $r$-graded modules over  the  polynomial ring with $r$ variables and coefficients in  the  field $K$.

\begin{point}{\bf Semisimplicity.}\label{pt semisimple}
A functor   $F\colon { \mathbf  N}^r\to \text{Vect}_K$ is called {\bf semisimple} if $F(v<u)$ is the zero homomorphism for any non-identity relation $v<u$.
For example, the unique functor $U_w\colon  { \mathbf  N}^r\to \text{Vect}_K$ such that $U_w(w)=K$  and $U_w(v)=0$ if $v\not =w$  is  semisimple.  A semisimple functor $F\colon { \mathbf  N}^r\to \text{Vect}_K$ is 
isomorphic to the direct sum $\oplus_{v\in   { \mathbf  N}^r}  (U_v\otimes F(v))$ and thus
two such functors are isomorphic if and only if they have isomorphic values.
%
\end{point}
\begin{point}{\bf Semisimplifications.}\label{pt semisimplif}
Consider a frame $F\colon { \mathbf  N}^r\to \text{Vect}_K$. For an object $v$ in ${ \mathbf  N}^r$ and $i$ in $ \mathbb   N$, define $\delta_i\colon \Delta F(v)_i\to \Delta F(v)_{i-1}$ to be the homomorphism:
\[\delta_i\colon  \bigoplus_{\substack{S\subset \{1,\ldots, r\}\\
|S|=i}} F(v-e_S)\to 
\bigoplus_{\substack{T\subset \{1,\ldots, r\}\\
|T|=i-1}} F(v-e_T)
\] 
 given by the matrix with the following coordinates (see~\ref{pt NQ} for the notation):
\[(\delta_i)_{T,S}:=\begin{cases}
0 &\ \text{if } T\not\subset S\\
(-1)^{a}F((v-e_S)<( v-e_T))& 
\begin{array}{l}\text{if } T\subset S 
 \text{ and where $a$ is the order}\\
 \text{in $S$ of the only element in  $S\setminus T$}\end{array}
\end{cases}\]
Note that $ \Delta F(v)_i=0$ if $i>r$.
A consequence of 
 the fact that $F$ is a functor is the equality 
$\delta_i\delta_{i-1}=0$. Thus,  for a given $v$ and  all $i$ in $ \mathbb   N$,
these homomorphisms define a chain complex  denoted by $\Delta F(v)$ and  called the {\bf Koszul complex}  of $F$ at $v$.  
 Note further that if $v\leq w$, then the homomorphism $\Delta F(v)\to\Delta F(w)$, induced by $F((v-e_S)\leq (w-e_S))$ for any 
$S\subset \{1,\ldots, r\}$, is a map of chain complexes.  
These maps in fact define a functor $\Delta F\colon { \mathbf  N}^r\to \text{Ch}_\geq(K)$ with values in the category of non-negative chain complexes $\text{Ch}_\geq(K)$.  

By taking   the $i$-th homology   we obtain  a frame $H_i(\Delta F)\colon{ \mathbf   N}^r\to  \text{Vect}_K$. Note that  if $v<w$, then $\Delta F(v<w)_i$ maps any summand $F(v-e_S)$ in $\Delta F(v)_i$ into the  
image of $\delta_{i+1}\colon \Delta F(w)_{i+1}\to \Delta F(w)_{i}$. This means that $H_i(\Delta F)(v<w)$ is the trivial homomorphism and consequently the  homology  functors $H_i(\Delta  F)\colon{\mathbb N}^r\to  \text{Vect}_K$    are all semisimple.  We therefore also call them {\bf semisimplifications} of $F$. For example:
\[H_i(\Delta  U_w)=\bigoplus_{\substack{S\subset \{1,\ldots, r\}\\
|S|=i}}U_{w+e_S}\ \ \ \ \ \ \ \ 
H_i(\Delta  K(w,-))=\begin{cases}
U_w& \text{ if } i=0\\
0& \text{ if } i>0
\end{cases}
\]
\[
H_i(\Delta \oplus_{v\in {\mathbb N}^r} (K(v,-)\otimes V_v))=\begin{cases}
\oplus_{v\in {\mathbb N}^r} U_v\otimes V_v& \text{ if } i=0\\
0& \text{ if } i>0
\end{cases}
\]
Thus if  $F$ is  free, then $F$ is isomorphic to $\oplus_{v\in {\mathbb N}^r} K(v,-)\otimes H_0(\Delta F(v))$
and we see that in  this case $H_0(\Delta F(v))$ is isomorphic to $H_0F(v)$ as defined in~\ref{pt free}. 

If $v$ is a minimal element in the  poset  ${ \mathbf   N}^r$ for  which $F(v)\not =0$, then $\Delta F(v)=F(v)$ and hence $H_0(\Delta F)(v)=F(v)$.
It follows that  $F\not =0$ if and only if $H_0(\Delta F)\not=0$. Thus the $0$-th semisimplification detects non triviality of a frame. 

Let $0\to F_0\to F_1\to F_2\to 0$ be  an exact sequence of frames.
As the Koszul complex is formed by taking direct sums and direct sums preserve exactness, we get en exact sequence of chain complexes
$0\to\Delta F_0\to\Delta F_1\to \Delta F_2\to 0$. By taking the homology we then obtain a long exact sequence of semisimple  frames:
\[\xymatrix{
             &  &  &
                 \ar@{-->} `r/8pt[d] `/10pt[l] `^dl[ll] `^r/3pt[dll] [dll]\\
             & H_1( \Delta F_0) \ar[r] & H_1( \Delta F_1) \ar[r] & H_1(\Delta F_2)  
             \ar@{->} `r/8pt[d] `/10pt[l] `^dl[ll] `^r/3pt[dll] [dll]
             \\
             & H_0( \Delta F_0) \ar[r] & H_0( \Delta F_1) \ar[r] & H_0(\Delta F_2)\ar[r] & 0
 }\]
 For example, let  $v< w$ in ${\mathbf N}^r$ and  consider the exact sequence of frames 
 $0\to K(w,-)\to K(v,-)\to [v,w)\to 0$ (see~\ref{pt bar}).  This leads to an exact sequence of  semisimple frames 
 $0\to H_1(\Delta  [v,w))\to H_0(\Delta K(w,-))\to H_0(\Delta K(v,-))\to H_0(\Delta  [v,w))\to 0$. Thus:
    \[H_i(\Delta  [v,w))=\begin{cases}
U_v& \text{ if } i=0\\
U_w& \text{ if } i=1\\
0& \text{ if } i>1
\end{cases}
\]

Since $H_0$ detects non triviality, it also  detects epimorphisms:  a natural transformation  $\phi\colon 
F_1\to F_2$ is an epimorphism if and only if   $H_0(\Delta \phi)\colon
H_0(\Delta F_1)\to H_0(\Delta F_2)$ is an epimorphism. 
Detection of epimorphisms and non triviality combine  with   the above long exact sequence can be used to show: 
\begin{itemize}
\item  $[g_s]_{s\in S}\colon \oplus_{s\in S}K(w_s,-)\to F$  is an epimorphism   if and only if  $H_0(\Delta [g_s]_{s\in S})$ is an epimorphism.
\item $\phi\colon P_0\to F$ is  a minimal cover (see~\ref{pt freecover}) if and only if $P_0$ is free and 
 $H_0(\Delta\phi)$  is an isomorphism.
 \item Any  frame $F$ admits a minimal resolution 
$\cdots\to P_1\to P_0\to  F\to  0$.
\item  If $\cdots\to P_1\to P_0\to  F\to  0$  is a minimal resolution, then 
(i) $H_nF(v)$ and  $H_n(\Delta F(v))$ are isomorphic, 
(ii)  $P_n$ is isomorphic to $\oplus_{v\in {\mathbb N}^r} K(v,-)\otimes H_n(\Delta F(v))$, (iii)
 $P_n=0$ for   $n>r$.
\end{itemize}

The main point of this paragraph is: the Koszul complex $\Delta F$ is  a very   effective tool for calculating the homology of a  frame as defined in~\ref{pt bettigeneral}.
\end{point}

\begin{point}{\bf Betti diagrams, numbers, and Euler characteristic.}\label{pt vgadfgv}
Since  polynomial rings over fields are  Noetherian, a subfunctor of a finitely generated frame is also finitely generated. It follows that if $F\colon { \mathbf  N}^r\to \text{Vect}_K$  is  finitely generated, then so are all the functors in  its  minimal resolution
$\cdots\to P_1\to P_0\to F\to 0$ and all its semisimplifications $H_n(\Delta F)$. 
It follows that $F$ is finitely generated if and only if it is of finite rank as defined in~\ref{pt bettigeneral}. 
For such an $F$ we can use the  Koszul complex to calculate 
its  $n$-th Betti diagram and Betti number  as follows:
\[\beta_n F(v)=\text{dim}\, H_n(\Delta F(v))\ \ \ \ \ \ \beta_n F =\sum_{v\in{ \mathbb  N}^r}\text{dim}\, H_n(\Delta F)(v)\]
In particular $\beta_0F$, also referred to as the rank of $F$,  is the minimal number of generators of $F$.

Since for $n>r$, $\beta_nF=0$, we can define the Euler characteristic for finitely generated  frames as follows:
\[\chi (F ):= \sum_{n=0}^r (-1)^n \beta_n F\]
For example:
\[
\beta_n K(w,-)(v)=\begin{cases}
1& \text{if } n=0\text{ and } v=w\\
0& \text{otherwise}
\end{cases}\ \ \ \ \ \ \ 
\beta_n K(w,-)=\begin{cases}
1& \text{if } n=0\\
0& \text{if } n>0
\end{cases}
\]
\[ \chi(K(w,-))=1\]

\[
\beta_n [u,w)(v)=\begin{cases}
1& \text{if } n=0\text{ and } v=u\\
1& \text{if } n=1\text{ and } v=w\\
0& \text{otherwise}
\end{cases}\ \ \ \ \ \ \ 
\beta_n [u,w)=\begin{cases}
1& \text{if } n\leq 1\\
0& \text{if } n>1
\end{cases}
\]
\[ \chi([u,w))=0\]

\[
\beta_nU_w(v)=\begin{cases}
1& \text{if } v=w+e_S \text{ where }S\subset\{1,\ldots, r\}\text{ and } |S|=n\\
0& \text{otherwise }
\end{cases}
\]
\[
\beta_n U_w={r \choose n}\ \ \ \ \   \ \ \ \ \ 
\chi(U_w)=0\]
\end{point}
\section{Tame  functors}\label{sec tame}
In this section we recall the definition and basic properties of tame functors introduced in~\cite{martina}. We also explain  how to calculate  homological invariants such as Betti diagrams, Betti numbers, and the minimal number of generators of tame functors.

\begin{point}{\bf Tameness.}\label{pt tamness}
 Choose  a positive rational number $\alpha$ called a resolution  and consider  two functors:
 $  \alpha\colon { \mathbf  N}^r\to { \mathbf  Q}^r$ which is the multiplication by $\alpha$
 that maps $(n_1,\ldots n_r)$ to $(\alpha n_1,\ldots \alpha v_r)$ and 
 $\lfloor \alpha^{-1}\rfloor\colon  { \mathbf  Q}^r\to { \mathbf  N}^r$ that maps
 $v=(v_1,\ldots v_r)$ to $\lfloor \alpha^{-1}v\rfloor:=(\lfloor {\textstyle \frac{v_1}{\alpha}}\rfloor,\ldots, \lfloor {\textstyle \frac{v_r}{\alpha}}\rfloor)$  where for a non-negative rational number $t$,  the symbol $\lfloor t\rfloor$  denotes the biggest natural number smaller or equal than  $t$.    
    Note that the composition $\lfloor \alpha^{-1}\rfloor \alpha\colon { \mathbf  N}^r\to { \mathbf  N}^r$ is the identity and the functor $\lfloor \alpha^{-1}\rfloor$ is constant on  all  sub-posets of the form:
 \[[\alpha n_1,\alpha (n_1+1))\times [\alpha n_2,\alpha (n_2+1))\times \cdots\times [\alpha n_r,\alpha (n_r+1))\subset { \mathbf  Q}^r\] 
 for any $r$-tuple  of natural numbers $(n_1,\ldots, n_r)$. 
 We call these sub-posets  right open $\alpha$-cubes.

A functor $G\colon { \mathbf  Q}^r\to \text{Vect}_K$ is called $\alpha$-{\bf tame} if it is isomorphic to the composition
$F\lfloor \alpha^{-1}\rfloor$ for some      $F\colon { \mathbf  N}^r\to \text{Vect}_K$ (called an $\alpha$-frame of $G$). Since   $ \lfloor\alpha^{-1}\rfloor \alpha$ is the identity, the frame  $F$  has to be  necessarily isomorphic to $G\alpha\colon { \mathbf  N}^r\to \text{Vect}_K$. It is then clear that $G\colon { \mathbf  Q}^r\to \text{Vect}_K$ is $\alpha$-tame if and only if its  restriction to
any  right open $\alpha$-cube  is isomorphic to a constant functor.  This happens if and only if
$G(\alpha\lfloor \alpha^{-1}v\rfloor \leq v)$ is an isomorphism for any $v$ in ${ \mathbf  Q}^r$.

A functor $G\colon { \mathbf  Q}^r\to \text{Vect}_K$ is  called {\bf tame}, if it is $\alpha$-tame for some $\alpha$ (called a resolution of $G$).
For example the free functor $K(w,-)\colon  { \mathbf  Q}^r\to \text{Vect}_K$ on one generator in degree $w=(w_1,\ldots,w_r)$  (see~\ref{pt free}) is tame. Let $n_i=0$ and $d_i=1$  if  $w_i=0$ and, in the case $w_i\not =0$, let 
$n_i$ and $d_i$ to be coprime natural numbers such that $w_i=\frac{n_i}{d_i}$. Let $d$ be a common multiple of $d_1,\ldots, d_r$. Then $K(w,-)$ is constant on any right open $1/d$-cube in ${ \mathbf  Q}^r$
and hence it is $1/d$-tame.

Tameness is preserved by finite direct sums (see~\cite[Corollary 5.3]{martina}). More generally, for an exact sequence  $0\to G_0\to G_1\to G_2\to 0$ of functors indexed by ${\mathbf  Q}^r$, if  two out of $G_0$, $G_1$, $G_2$ are  tame, then so is
the third one.   For example, for any $v\leq w$ in ${\mathbf Q}^r$, the functor $[v,w)$ (see~\ref{pt bar}) is tame. 
Tameness however  is not preserved in general by  subfunctors  or  quotients. For example
 consider $K((0,0),-)\colon { \mathbf  Q}^2\to \text{Vect}_K$.  Its   subfunctor $G_0\subset K((0,0),-)$ given by:
 \[G_0(v_1,v_2)=\begin{cases}
 0 & \text{if } v_1+v_2<1\\
 K &\text{if } v_1+v_2\geq 1
 \end{cases}\]
 is not tame. Neither is the quotient $K((0,0),-)/G_0$. Note  that in this example $G_0$ is not finitely generated. 
If $\phi\colon G_0\to G_1$ is a natural transformation of tame functors, then according to~\cite[Proposition 5.2]{martina} $\text{ker}(\phi)$, 
$\text{coker}(\phi)$, and $\text{im}(\phi)$ are also tame. 
It follows that a finitely generated subfunctor of a tame functor is always tame.  Moreover a subfunctor
of a finitely generated and tame functor is tame if and only if this subfunctor is finitely generated.
\end{point}

\begin{point}{\bf Betti diagrams, numbers, and Euler characteristic.}\label{pt bettieuler}
Consider an  $\alpha$-tame functor  $G\colon { \mathbf  Q}^r\to \text{\rm Vect}_K $. 
Let  $\cdots\to P_n\to\cdots \to P_0\to G\alpha\to 0$ be a minimal free resolution of  its frame
$G\alpha\colon{\mathbf N}^r\to {\rm Vect}_K$. According to~\ref{pt semisimplif}, the functor  $P_n$
 is isomorphic to $\oplus_{v\in{\mathbb N}^r}(K_{{ \mathbf  N}^r}(v,-)\otimes H_n\Delta(G\alpha)(v)) $.  Precomposing with $\lfloor \alpha^{-1}\rfloor\colon  { \mathbf  Q}^r\to { \mathbf  N}^r$ is 
a faithful  and  exact operation. Moreover $K_{{ \mathbf  N}^r}(v,-)\lfloor \alpha^{-1}\rfloor$ is isomorphic to
$K_{{ \mathbf  Q}^r}(\alpha v,-)$. These facts imply that: 
\[\cdots\to P_n\lfloor \alpha^{-1}\rfloor\to\cdots\to  P_0\lfloor \alpha^{-1}\rfloor\to G\alpha\lfloor \alpha^{-1}\rfloor=G\to 0\] is a minimal resolution of $G$  and  $P_n\lfloor \alpha^{-1}\rfloor$
is isomorphic to:
\[\oplus_{v\in{\mathbb N}^r}(K_{{ \mathbf  Q}^r}(\alpha v,-)\otimes H_n\Delta(G\alpha)(v))\]
  We can use  these observations to conclude:
\begin{itemize}
\item Any $\alpha$-tame functor has  a minimal free resolution by $\alpha$-tame functors. The length of such a resolution does not exceed $r$.
\item  $H_nG(u)=\begin{cases}
H_n\Delta(G\alpha)(v) &\text{if }u=\alpha v\\
0&\text{otherwise}
\end{cases}$
\item $G$ is finely generated if and only if   $G\alpha$ is finitely generated.
\item  $G$  is finitely generated if and only if it is of finite rank (see~\ref{pt bettigeneral}).
\item If $\phi\colon G_0\to G_1$ is a natural  transformation between finitely generated tame functors, then
$\text{ker}(\phi)$, $\text{im}(\phi)$, and $\text{coker}(\phi)$ are also finitely generated tame functors.
\item If $G$ is finitely generated, then:
\[\beta_nG(u)=\begin{cases}
\beta_n(G\alpha)(v)=\text{dim}\,H_n\Delta (G\alpha)(v) &\text{if }u=\alpha v\\
0& \text{otherwise}
\end{cases}\]
\[\beta_n G=\sum_{v\in { \mathbb  N}^r} \text{dim}\, H_n(\Delta G\alpha)(v)
\]
\item  Let $G$ be finitely generated and let $\{g_s\in G(v_s)\}_{s=1}^{n}$  be a minimal set of generators for $G$.
Then the number of generators in this set whose coordinate is $v$ (i.e.\  generators that belong to $G(v)$) is given by $\beta_0G(v)$. In particular  $n=\beta_0G$. Moreover the set of coordinates of the generators coincides  with $\text{supp}(G)$.
\end{itemize}

A finitely generated tame functor is of finite rank and the length  of its minimal free resolution   is  not exceeding $r$.  For such a functor $G$, we can define its Euler characteristic as:
 \[ \chi (G ):=\sum_{n=1}^{r}(-1)^n\beta_n G\]
\end{point}

Let the symbol ${\mathbb T}({ \mathbf  Q}^r, \text{Vect}_K)$  denote the set of finitely generated tame functors.
The functions $\beta_n, |\chi| \colon {\mathbb T}({ \mathbf  Q}^r, \text{Vect}_K)\to{\mathbb N}$ 
(the $n$-th Betti number and the absolute value of the Euler characteristic)
 are the discrete invariants we plan stabilize (see section~\ref{sec:hier}). The last ingredient 
needed for the hierarchical stabilization is a choice of a 
 pseudometric on $\mathbb T({ \mathbf  Q}^r, \text{Vect}_K)$. Our next step is to recall how
to construct such metrics using  so called noise systems~\cite{martina}, which  is the subject of the next section.

\section{Noise systems}\label{sec noise}
Recall from \cite[Definition 6.1]{martina} that a \textbf{noise system} 
 is a sequence    $\mathcal{C}=\{\mathcal{C}_\epsilon\}_{\epsilon \in {\mathbb Q}}$ of subsets of
   ${\mathbb T}({ \mathbf  Q}^r, \text{Vect}_K)$   indexed by non-negative rational numbers, called  components, such that:
\begin{enumerate}
\item the zero functor belongs to $ \mathcal{C_\epsilon}$ for any $\epsilon$ in ${\mathbb Q}$,
\item if $ \tau \leq  \epsilon$, then $\mathcal{C_\tau}\subseteq  \mathcal{C_\epsilon}$,
\item if $0 \to G_0 \to G_1 \to G_2 \to 0$ is an exact sequence of finitely generated tame functors, then:
\begin{enumerate}
\item if $G_1$ is in $ \mathcal{C_\epsilon}$, then so are $G_0$ and $G_2$,
\item if $G_0$ is in $ \mathcal{C_\tau}$ and $G_2$ is in $ \mathcal{C_\epsilon}$, then $G_1$ is in $ \mathcal{C_{\epsilon+\tau}}$.
\end{enumerate}
\end{enumerate}

A noise system $\{\mathcal{C}_\epsilon\}_{\epsilon \in {\mathbb Q}}$ is {\bf closed under direct sums} if,  for any $\epsilon$, the set   $ \mathcal{C_\epsilon}$ is closed under direct sums: $G_0$ and $G_1$ belongs to 
$ \mathcal{C_\epsilon}$ if and only if their direct sum $G_0\oplus G_1$ belongs to $ \mathcal{C_\epsilon}$.
Noise systems closed under direct sums are quite special. They are determined by  cyclic functors (see~\ref{pt freecover}):

\begin{prop}\label{prop cyclcialldirsum}
Let $\mathcal{C}=\{\mathcal{C}_\epsilon\}_{\epsilon \in {\mathbb Q}}$  and $\mathcal{D}=\{\mathcal{D}_\epsilon\}_{\epsilon \in {\mathbb Q}}$ 
be noise systems closed under direct sums. Then $\mathcal{C}_\epsilon=\mathcal{D}_\epsilon$ if and only if they contain the same cyclic functors.
\end{prop}
\begin{proof}
Assume $\mathcal{C}_\epsilon$ and $\mathcal{D}_\epsilon$  contain the same cyclic functors.
Let $G$ be in $\mathcal{C}_\epsilon$ and choose its minimal set of generators
$\{g_s\in G(v_s)\}_{s=1}^{n}$. Then, for any $s$, the cyclic subfunctor of $G$ generated by $g_s$ 
also belongs to $\mathcal{C}_\epsilon$. By assumption, these cyclic subfunctors  belong to
$\mathcal{D}_\epsilon$ and so does their direct sum. As a quotient of this direct sum, the functor  $G$ is therefore  a member of $\mathcal{D}_\epsilon$. This shows the inclusion $\mathcal{C}_\epsilon\subset \mathcal{D}_\epsilon$.
By symmetry   we then get an equality $\mathcal{C}_\epsilon= \mathcal{D}_\epsilon$.
\end{proof}

We think about the component $\mathcal{C}_\epsilon$ as the closed disc of radius $\epsilon$ around the zero functor and  regard  its  members  as  functors which are $\epsilon $ small. This notion of smallness can be extended to a notion of proximity among arbitrary finitely generated tame functors $G_0$ and $G_1$ as follows. First, recall~\cite[Section 8]{martina} 
that a natural transformation $\psi$ between functors in  ${\mathbb T}({ \mathbf  Q}^r, \text{Vect}_K)$ is called  an \textbf{$\epsilon$-equivalence} if $\text{ker}(\psi)$ is in $ \mathcal{C}_a$, $\text{coker}(\psi)$ is in $\mathcal{C}_b$, and $a+b\leq \epsilon$. Second, define  $G_0$ and $G_1$ to be  $\epsilon$-\textbf{close} if there exists  $G_2$ in  ${\mathbb T}({ \mathbf  Q}^r, \text{Vect}_K)$ and natural transformations $G_0\leftarrow G_2\colon \phi$ and $\psi\colon G_2\to G_1$ such that $\phi$ is a $\tau$-equivalence, $\psi$ is a $ \mu$-equivalence, and $\tau+ \mu\leq \epsilon$.
Finally set:
\[
d(G_0, G_1) := \begin{cases} \text{inf}\, \{\epsilon \ |\  \text{$G_0$ and $G_1$ are $\epsilon$-close}\} &\text{if $G_0$ and $G_1$ are close for some $\epsilon$} \\
\infty &  \text{otherwise }
\end{cases}
\]
A key result of \cite{martina} states that the above defined $d$ is a  pseudometric on the set ${\mathbb T}({ \mathbf  Q}^r, \text{Vect}_K)$. For  $F$ in ${\mathbb T}({ \mathbf  Q}^r, \text{Vect}_K)$, denote the set of all functors  that are $\epsilon$-close to $F$ by  $B(F, \epsilon)$. This is the closed disc around $F$ of radius $\epsilon$ with respect to the metric $d$. For example $B(0, \epsilon)= \mathcal{C_\epsilon}$.

We now have all three  ingredients needed  for  hierarchical stabilization (see Section~\ref{sec:hier}):
\begin{enumerate}
\item[(i)] a set ${\mathbb T}({ \mathbf  Q}^r, \text{Vect}_K)$,
\item[(ii)] discrete invariants $\beta_n,\  |\chi |\colon{\mathbb T}({ \mathbf  Q}^r, \text{Vect}_K)\to {\mathbb N}$,
\item[(iii)] a pseudometric (induced by a noise system) on ${\mathbb T}({ \mathbf  Q}^r, \text{Vect}_K)$.
\end{enumerate}
For any choice of noise system $\{{\mathcal C}_{\epsilon}\}_{\epsilon\in{\mathbb Q}}$,  by minimizing over the discs $B(G, \tau)$, we   obtain  $1$-Lipschitz functions (see Section~\ref{sec:hier}):
\[
\widehat{\beta}_n,\    \widehat{|\chi|}\colon {\mathbb T}({ \mathbf  Q}^r, \text{Vect}_K)\to \text{\rm Mult}(  { \mathbb   Q})
\]
For example, let $v\leq w$ be in ${\mathbf Q}^r$. Then:
\[
\widehat{\beta}_0K(v,-)(\tau)=\begin{cases}
1 &\text{if } K(v,-)\not\in {\mathcal S}_{\tau}\\
0 &\text{if } K(v,-)\in {\mathcal S}_{\tau}
\end{cases}\ \ \ \ 
\widehat{\beta}_0[v,w)(\tau)=\begin{cases}
1 &\text{if } [v,w)\not\in {\mathcal S}_{\tau}\\
0 &\text{if } [v,w)\in {\mathcal S}_{\tau}
\end{cases}
\]

The rest of the paper is devoted to explaining  a strategy of how to calculate the invariant $\widehat{\beta}_0$ with respect to   so  called  \textit{simple}
noise systems. 

\section{Simple noise systems}\label{sec:simplenoise}
In this section we define a subclass of noise systems which we call \textit{simple noise systems}, containing e.g. the \textit{standard noise system} of \cite{martina} (see Example \ref{ex standardpc}). The noise systems in this subclass have extra properties that allow us to perform computations on them. The most important property of simple noise systems is that they can be defined constructively using \textit{persistence contours} (see Section \ref{sec:pers}).

Recall that the $0$-th Betti number $\beta_0G$ (also called the rank) of a finitely
generated tame functor $G$ is equal to   the size of any minimal set of generators of $G$.
Let    $\mathcal{C}=\{\mathcal{C}_\epsilon\}_{\epsilon \in {\mathbb Q}}$ be a noise system and $G\colon { \mathbf  Q}^r\to\text{Vect}_K$ a finitely
generated tame functor. 
By definition, to calculate the value of   $\widehat{\beta}_0 G\colon{\mathbb Q}\to{\mathbb N}$ at $\tau$, we need to find the smallest  $0$-th Betti number
among the functors in the disc $B(G,\tau)$.  In general, the set $B(G,\tau)$ might be infinite even in the case the underlying field is finite. Our strategy 
is to understand under what circumstances we only need to calculate the $0$-th Betti numbers of finitely many functors in order to identify  $\widehat{\beta}_0 G({\tau})$.
For that purpose consider the following  collection of subfunctors of $G$:
\[B_{\subseteq}(G, \tau):=\{ G_0 \ |\   G_0\text{ is a tame subfunctor of $G$ such that }G/G_0 \text{ belongs to } \mathcal{C}_\tau \}\]
This collection may or may not have any minimal element with respect to the  inclusion relation. However, in the case that $\mathcal{C}_\tau$ is closed under direct sums, if some minimal element in 
$B_{\subseteq}(G, \tau)$ exists, than it has to be unique: 

\begin{prop}\label{prop:exmin} 
Assume  $\mathcal{C}_\tau$ is closed under direct sums.
Let $F$ and $G$  be finitely generated tame functors.
\begin{enumerate}
\item   If $G_0$ and $G_1$ are minimal elements in $B_{\subseteq}(G, \tau)$, then $G_0=G_1$.
\item Let $G[\tau]$ and $F[\tau]$ be minimal elements in $B_{\subseteq}(G, \tau)$ and $B_{\subseteq}(F, \tau)$ respectively.  Then any  natural transformation  $\phi\colon G\to F$ maps $G[\tau]$ into $F[\tau]$. Moreover, if 
$\phi$ is an epimorphism, then so is its restriction $\phi\colon G[\tau]\to F[\tau]$.
\item If  $F[\tau]$, $G[\tau]$ and $ (F\oplus G)[\tau]$ exist, then  $(F\oplus G)[\tau]=F[\tau]\oplus G[\tau]$.
\end{enumerate}
\end{prop}
\begin{proof}
\noindent
(1):\quad
The functor $G/G_0\oplus G/G_1$ belongs to $\mathcal{C}_\tau$ as the noise is  closed under direct sums. The image of the homomorphism $\pi\colon G\to G/G_0\oplus G/G_1$, given by the projections, 
is therefore also   a member of $\mathcal{C}_\tau$ and consequently  $\text{ker}(\pi)=G_0\cap G_1$  belongs  to $B_{\subseteq}(G, \tau)$. Minimality of  \medskip $G_0$ and $G_1$  implies $G_0=G_0\cap G_1=G_1$. \smallskip

\noindent
(2):\quad Since $F/(F[\tau])$ belongs to ${\mathcal C}_{\tau}$, then so does the image of the composition of $\phi\colon G\to F$ and the quotient $F\to F/(F[\tau])$. The kernel of this composition therefore belongs to
$B_{\subseteq}(G, \tau)$. By construction $\phi$ maps this kernel into $F[\tau]$. To finish the argument just notice that, by minimality,
$G[\tau]$ is a subfunctor of this kernel. 

If  $\phi$ is an epimorphism, then  $F/\phi(G[\tau])$  is a quotient of  $G/(G[\tau])$.
Consequently  $F/\phi(G[\tau])$  belongs to $\mathcal{C}_\tau$ and hence we have an inclusion
$F[\tau]\subset \phi(G[\tau])$.   This together with what   we already have shown implies equality
$F[\tau]= \phi(G[\tau])$.
\smallskip

\noindent
(3):\quad Since  $\mathcal{C}_\tau$ is closed under direct sums, we have  $(F\oplus G)[\tau]\subseteq F[\tau]\oplus G[\tau]$.
By part (2), the inclusions $F\subset F\oplus G\supset G$ induce inclusions $F[\tau]\subset (F\oplus G)[\tau]\supset G[\tau]$. That gives $F[\tau]\oplus G[\tau]\subset (F\oplus G)[\tau]$.
\end{proof}

In the case $\mathcal{C}_\tau$ is closed under direct sums, if it exists, we denote this minimal element in $B_{\subseteq}(G, \tau)$ by $G_{\mathcal{C}}[\tau]\subset G$, or simply by $G[\tau]\subset G$ if the noise system is fixed.  The subfunctor $G[\tau]\subset G$ is also called the {\bf ${\tau}$-shift} of $G$.
Note that in the case $G[\tau]\subset G$ exists, since $\mathcal{C}_\tau$ is closed under taking quotients,
a  tame subfunctor $G_0\subseteq G$ has the property that $G/G_0$ belongs to $\mathcal{C}_\tau$  if and only if $G[\tau]\subseteq G_0$.  Thus in this case 
$B_{\subseteq}(G, \tau)$ can be identified with the set:
\[\{G_0\ |\ G_0\text{ is a tame subfunctor of $G$ such that  }G[\tau]\subseteq G_0\subseteq G\}\]

We are now ready to define what we call a simple noise system. 

\begin{Def}\label{def simplenoise}
A noise system $\mathcal{C}=\{\mathcal{C}_\epsilon\}_{\epsilon \in {\mathbb Q}}$ is called {\bf simple} if:
\begin{itemize}
\item it is closed under direct sums,
\item for any  finitely
generated tame functor $G\colon { \mathbf  Q}^r\to\text{\rm Vect}_K$ and any  $\tau$  in ${\mathbb Q}$, the set 
$B_{\subseteq}(G, \tau)$ contains the minimal element $G[\tau]\subset G$,
\item $\beta_{0} G[\tau]\leq \beta_0 G$ for any $\tau$  in ${\mathbb Q}$.
\end{itemize}
\end{Def}

Recall  that for a subset $X\subset {\mathbb Q}^r$ and an element $v$ in ${\mathbb Q}^r$, we write $X\leq v$ if $w\leq v$ for any $w$ in $X$ (see~\ref{pt NQ}). 

\begin{thm}\label{thm calsimple}
Let $\{\mathcal{C}_\epsilon\}_{\epsilon \in {\mathbb Q}}$ be a simple noise system and $G\colon{\mathbf Q}^r\to \text{\rm Vect}_K$  a finitely generated tame  functor.  Let  $v$ be in ${\mathbb Q}^r$ such  that $\text{\rm supp}(G[\tau]) \leq v$.
Then:
\[\widehat{\beta}_0 G(\tau)=\text{\rm min}\left\{
\begin{array}{l|c}
\multirow{2}{*}{$\beta_0F$} &F 
\text{ is a tame subfunctor of  } G\text{ such that}\\
 & G[\tau]\subset F\subset G\text{ and }\text{\rm supp}(F)\leq v
\end{array}\right\}
\]
%
\end{thm}
\begin{proof}
The proof is very similar to the proof of~\cite[Proposition 10.1]{martina}. 
Let $n:=\widehat{\beta}_0 G(\tau)$  
and $G\leftarrow G_1:\!\phi$ and $\psi\colon G_1\to G_2$
be natural transformations between finitely generated tame functors such that $\beta_0 G_2=n$ and:
\[\text{ker}(\phi)\in \mathcal{C}_a,\ \ \ \text{coker}(\phi)\in \mathcal{C}_b,\ \ \ \text{ker}(\psi)\in \mathcal{C}_c,\ \ \ \text{coker}(\psi)\in \mathcal{C}_d,\ \ \ a+b+c+d\leq \tau\]
Consider the subfunctor $G_2[d]\subset G_2$ and let  $\{g_s\in G_2[d](v_s)\}_{s=1}^{k}$
be  a minimal set of generators.   Since the considered noise system is simple, $k\leq n$. By  the minimality of  $G_2[d]\subset G_2$, we have  $G_2[d]\subset \text{im}(\psi)$.
 Thus we can choose elements  $\{g'_s\in G_1(v_s)\}_{s=1}^{k}$ 
that are mapped via $\psi$ to the chosen minimal generators of  $G_2[d]$. Let $G'_{1}\subset G_1$ be the subfunctor generated by these elements and $G'\subset G$ be its image $\phi(G'_1)$. All these functors are finitely generated and tame  and they can be arranged into a commutative diagram where the indicated natural transformations are monomorphisms and epimorphisms:
\[\xymatrix{
G'\ar@{^(->}[d] &G'_1\ar@{^(->}[d]\ar@{->>}[l]\ar@{->>}[r] & G_2[d] \ar@{^(->}[d]\\
G & G_1\ar[l]_{\phi} \ar[r]^{\psi}& G_2
}\]
We claim that $G/G'$ belongs to $\mathcal{C}_{b+c+d}\subset \mathcal{C}_\tau$. This is a consequence of the additivity property of  noise systems applied to the  following exact sequences:
\[G_1/G'_1\to G/G'\to\text{coker}(\phi)\to 0\]
\[0\to \text{ker}(\psi)/(\text{ker}(\psi)\cap G'_1)\to G_1/G'_1\to G_2/G_2[d]
\]

The above claim  implies  $G[\tau]\subset G'$. 
Let $\{x_s\in G'(w_s)\}_{s=1}^{l}$ be a minimal set of generators of $G'$.  
Define $G''\subset G'$ to be generated by $\{x_s\ |\ 1\leq s\leq l\text{ and } w_s\leq v\}$.
Since $\text{supp}(G[\tau])\leq v$,  for  degree reasons  $G[\tau]$ is included in $G''$.
By construction  $\beta_0 G''\leq \beta_0 G'\leq k$ and $k\leq n$, which gives:
\[\widehat{\beta}_0 G(\tau) \geq  \text{\rm min}\left\{
\begin{array}{l|c}
\multirow{2}{*}{$\beta_0F$} &F 
\text{ is a tame subfunctor of  } G\text{ such that}\\
 & G[\tau]\subset F\subset G\text{ and }\text{\rm supp}(F)\leq v
\end{array}\right\}\] 
As the reverse inequality $\leq$ is obvious, the  equality of  the theorem is  proven.
\end{proof}

The set of tame subfunctors of $G$ considered in  Theorem~\ref{thm calsimple} can still be infinite. 
 To reduce the calculation of  $\widehat{\beta}_0 G ({\tau})$ to a finite set we need an additional  step.  

\begin{cor}\label{cor redfinite}
Let $\{\mathcal{C}_\epsilon\}_{\epsilon \in {\mathbb Q}}$ be a simple noise system and $G\colon{\mathbf Q}^r\to \text{\rm Vect}_K$  a  finitely generated tame  functor.  Choose a rational number $\alpha$ such that
both $G$ and its shift $G[\tau]$ are $\alpha$-tame. Choose $v$ in ${\mathbb N}^r$ so that $\text{\rm supp}(G[\tau]\alpha)\leq v$. Then:
\[\widehat{\beta}_0 G(\tau)=\text{\rm min}\left\{
\begin{array}{l|c}
\multirow{2}{*}{$\beta_0F$} &F\colon{\mathbf N}^r\to \text{\rm Vect}_K 
\text{ is a subfunctor of $G\alpha$ s.t.: }\\
 & G[\tau]\alpha\subset F\subset G\alpha\text{ and }\text{\rm supp}(F)\leq v
\end{array}\right\}
\]
\end{cor}
\begin{proof}
Since $\text{\rm supp}(G[\tau]\alpha)\leq v$, then $\text{\rm supp}(G[\tau])\leq  \alpha v$. We can then use Theorem~\ref{thm calsimple} to obtain a  finitely generated tame functor
 $F'\colon{\mathbf Q}^r\to \text{Vect}_K$ such that $ G[\tau]\subset F'\subset G$,
 $\text{supp}(F')\leq \alpha v$, and $\beta_0F'=\widehat{\beta}_0 G(\tau)$. By restricting along
 $\alpha\colon{\mathbf N}^r\to {\mathbf Q}^r$ we get a sequence of frames $G[\tau]\alpha\subset F'\alpha\subset G\alpha$.  Since $\text{supp}(F')\leq \alpha v$, we also have $\text{supp}(F'\alpha)\leq  v$. 
Finally, as the  restriction preserve the rank,  $\beta_0F'\alpha=\widehat{\beta}_0 G(\tau)$. That shows 
 that the left side of the equality in the statement of the corollary can not be smaller than the right side.
 As the opposite inequality is clear, the corollary is proven.
 \end{proof}

 A consequence of Corollary~\ref{cor redfinite} is that  to determine the value $\widehat{\beta}_0 G(\tau)$  for a finite field $K$ we need to perform only finitely many operations as the following set is finite:
  \[\{ F\colon{\mathbf N}^r\to \text{\rm Vect}_K\ |\ F  \text{ is a subfunctor of $G\alpha$ such that }  \text{\rm supp}(F)\leq v \}\]

Definition~\ref{def simplenoise} describes a simple noise system  in terms of certain  global properties of its components. 
For   implementation and calculation purposes a constructive description of simple noise systems is needed. A description that  would also lead to  an explicit construction  of the subfunctor $G[\tau]\subset G$,
so we can use Theorem~\ref{thm calsimple} and Corollary~\ref{cor redfinite} to calculate the stabilization of $\beta_0F$. This is what we will focus on in the next section, culminating in Theorem \ref{thm:contour bijection}, showing that there is a bijection between simple noise systems and persistence contours.

\section{Persistence contours}\label{sec:pers}
Recall (see~\ref{pt NQ})
that ${ \mathbf   Q}^r_{\infty}$ is  the poset obtained 
from  ${ \mathbf   Q}^r$ by adding an extra object $\infty$ such that $v\leq\infty$ for any  $v$  in ${ \mathbb   Q}^r$.
For any $v$ in ${ \mathbb   Q}^r_{\infty}$, we set $v+\infty:=\infty$.

\begin{notation}\label{not adfaergsfhdg}
Let  $\{\mathcal{C}_\epsilon\}_{\epsilon \in {\mathbb Q}}$ be a simple noise system. 
For   $(v,\epsilon) $   in ${ \mathbb  Q}^r_{\infty}\times {\mathbb Q}$, define $C(v,\epsilon)$ in ${ \mathbb  Q}^r_{\infty}$ as follows:
\begin{itemize}
\item 
 $C(\infty,\epsilon):=\infty$. 
\item   Consider the functors $K(v,-)[\epsilon]\subset K(v,-)$.
Because the  noise system  is simple,  $\beta_0 K(v,-)[\epsilon]\leq \beta_0 K(v,-)=1$. Thus either $K(v,-)[\epsilon]$ is trivial or it is free on one generator. 
 If $K(v,-)[\epsilon]=0$, then we set 
$C(v,\epsilon):=\infty$. If $K(v,-)[\epsilon]$ is free on one generator, 
define $C(v,\epsilon)$ to be the object  in ${ \mathbf  Q}^r$  for which  $K(v,-)[\epsilon]$ is isomorphic to 
 $K(C(v,\epsilon),-)$. 
\end{itemize}
\end{notation}
 
We claim that if $w\leq  v$ in ${ \mathbb  Q}^r_{\infty}$ and $\tau\leq \epsilon$ in ${ \mathbb  Q}$, then
$C(w,\tau)\leq C(v,\epsilon)$. This means that  the association $(v,\epsilon)\mapsto C(v,\epsilon)$ is in fact a functor $C\colon {\mathbf Q}^r_{\infty}\times {\mathbf Q}\to {\mathbf Q}^r_{\infty}$. The claim is obvious if $v=\infty$ or $v$ is in ${\mathbb Q}^r$ and $K(v,-)[\epsilon]=0$. In the remaining case of    $v$ being  in ${\mathbb Q}^r$ and $K(v,-)[\epsilon]$ being   non-trivial, 
since  ${\mathcal C}_\tau\subset {\mathcal C}_\epsilon$, we have  monomorphisms:
 \[\xymatrix{
 K(C(v,\epsilon),-)= K(v,-)[\epsilon]\ar@{^(->}[r] & K(v,-)[\tau]\ar@{^(->}[r] \ar@{^(->}[d] & K(v,-)\ar@{^(->}[d] \\
  &  K(w,-)[\tau]\ar@{^(->}[r] &  K(w,-)
 }\]
The functor $K(w,-)[\tau]$  is therefore  non-trivial. The relation $C(w,\tau)\leq C(v,\epsilon)$ is then a consequence of the fact that   
$ K(C(v,\epsilon),-)$ is a subfunctor of $ K(C(w,\tau),-)$.
 
 The functor $C\colon {\mathbf Q}^r_{\infty}\times {\mathbf Q}\to {\mathbf Q}^r_{\infty}$ is not arbitrary.
 It satisfies the following two additional properties   for any $v$ in $ {\mathbb Q}^r_{\infty}$ and any $\epsilon $ and $\tau$ in  	${\mathbb 	Q}$:
 \[v\leq C(v,\epsilon),\ \ \ \ \ \ \ C(C(v,\epsilon),\tau)\leq C(v,\epsilon+\tau)\]
 The first property is  clear when $C(v,\epsilon)=\infty$.
 If $C(v,\epsilon)\not=\infty$, then $v\not=\infty$. In this case the  relation $v\leq C(v,\epsilon)$ is a consequence of    $K(C(v,\epsilon),-)$ being a subfunctor of $K(v,-)$. 
 The second property is also clear if $C(v,\epsilon+\tau)=\infty$. If $C(v,\epsilon+\tau)<\infty$, then we have a sequence
 of monomorphisms:
 \[
K(C(v,\epsilon+\tau),-)= K(v,-)[\epsilon+\tau]\subset  K(v,-)[\epsilon][\tau]\subset  K(v,-)[\epsilon]\subset  K(v,-)
\]
Thus $K(v,-)[\epsilon][\tau]$ is non trivial, and according to our notation, isomorphic to 
$K(C(C(v,\epsilon),\tau),-)$. As  this functor contains $K(C(v,\epsilon+\tau),-)$, we get the claimed
relation $C(C(v,\epsilon),\tau)\leq C(v,\epsilon+\tau)$.

We formalize these properties in the following definition:
\begin{Def}
	\label{def:func}
	A \textbf{persistence contour}  is a functor 
	$C\colon  {\mathbf Q}^r_{\infty}\times {\mathbf Q}\to {\mathbf Q}^r_{\infty}$ such that, for any $v$ in $ {\mathbb Q}^r_{\infty}$ and any $\epsilon $ and $\tau$ in  	${\mathbb 	Q}$: 
			\begin{enumerate}
			\item $v\leq C(v,\epsilon)$,
			\item $C(C(v,\epsilon),\tau)\leq C(v,\epsilon+\tau)$.
		\end{enumerate}
\end{Def}

What we call persistence contours were introduced, in a different setting, by Bubenick et. al. under the name superlinear families \cite{Bubenik2015}.
We have seen how a simple noise system leads to a persistence contour.
It turns out that this procedure can be reversed. 

\begin{notation}\label{not svcadfvadfva}
Let $D\colon {\mathbf Q}_{\infty}^r\times {\mathbf Q}\to {\mathbf Q}_{\infty}^r$ be a persistence contour. 
\begin{itemize}
\item
For $\epsilon$ in ${\mathbb Q}$ define
${\mathcal D}_{\epsilon}\subset {\mathbb T}({ \mathbf  Q}^r, \text{Vect}_K)$ to be the collection of the  finitely generated tame functors $G\colon { \mathbf  Q}^r\to\text{Vect}_K$ for which   $G\big(v\leq D(v,\epsilon)\big)$  is the zero homomorphism whenever  $D(v,\epsilon)\not=\infty$.
\item
Let $G$ be a finitely generated tame functor. Choose its  minimal set of generators $\{g_s\in G(v_s)\}_{s=1}^{n}$ and $\tau$ in ${\mathbb Q}$.  For any $s$ such that  $D(v_s,\tau)\not=\infty$, set $h_s:=G\big(v_s\leq D(v_s,\tau)\big)(g_s)\in G(D(v_s,\tau))$.
Define $G[\tau]\subset G$ to be the subfunctor generated by  $\{h_s\ | \ 1\leq s\leq n\text{ and } D(v_s,\tau)\not=\infty\}$. 
\end{itemize}
\end{notation}
If $D(v,\tau)=\infty$, then since $D$ is a functor, $D(w,\tau)=\infty$ for any $v\leq w$. 
It follows that $K(v,-)$ belongs to ${\mathcal D}_{\tau}$ if and only if $D(v,\tau)=\infty$.
Furthermore:
\[K(v,-)[\tau]=\begin{cases}
0& \text{if } D(v,\tau)=\infty\\
K\big(D(v,\tau),-\big) &\text{if } D(v,\tau)\not=\infty
\end{cases}\]

Note also that if $C,D\colon {\mathbf Q}_{\infty}^r\times {\mathbf Q}\to {\mathbf Q}_{\infty}^r$ are persistence contours such that $C(v,\epsilon)\leq D(v,\epsilon)$ for any $(v,\epsilon)$ in ${\mathbf Q}_{\infty}^r\times {\mathbf Q}$, then ${\mathcal C}_\epsilon\subset {\mathcal D}_\epsilon$  for any $\epsilon$ in ${\mathbb Q}$.

 A priori the subfunctor $G[\tau]\subset G$, defined in~\ref{not svcadfvadfva},  depends on the choice of the generators of   $G$. In Proposition~\ref{prop svfwrtg}
we are going to show that $G[\tau]\subset G$ is the minimal subfunctor,  with respect to inclusion, for which $G/(G[\tau])$ belongs to 
${\mathcal D}_{\tau}$. 


\newcommand{\DD}{{\mathcal D}}

\begin{prop}\label{prop svfwrtg}
The sequence $\D=\{\DD_\epsilon\}_{\epsilon\in {\mathbb Q}}$ is a simple noise system.
\end{prop}
\begin{proof}
\noindent
{\bf Noise conditions.}
It is clear that the zero  functor belongs to  $\DD_\epsilon$ for any $\epsilon$.
Let $\tau\leq \epsilon$,  $G$ be in $\DD_{\tau}$, and   $v$ be an element in ${\mathbb Q}^r$ for which   $D(v,\epsilon)\not=\infty$. Since $D(v,\tau)\leq D(v,\epsilon)$, we also have $D(v,\tau)\not=\infty$. The  
 homomorphism  $G(v\leq D(v,\tau))$ is therefore  trivial, and hence so is the composition  $G\big(v\leq D(v,\epsilon)\big)=G\big(D(v,\tau)\leq D(v,\epsilon)\big)\circ G\big(v\leq D(v,\tau)\big)$. Thus $G$ is in 
 $\DD_{\epsilon}$. This shows $\DD_{\tau}\subset \DD_{\epsilon}$.
The first two requirements for  a noise system are therefore  satisfied (see Section~\ref{sec noise}).

Let $0\to G_0\to G_1\to G_2\to 0$ be an exact sequence. Assume  $G_1$ belongs to $ \DD_\epsilon$. By definition  $G_1(v\leq D(v, \epsilon))$ is the zero  homomorphism for any $v$ for which $D(v,\epsilon)\not=\infty$. 
Since $G_1(v)\to G_2(v)$ is an epimorphism and  $G_0(D(v,\epsilon))\to G_1{}(D(v,\epsilon))$  is a monomorphism, the homomorphisms $G_2(v\leq D(v,\epsilon))$ and  $G_0(v\leq D(v,\epsilon))$ have to be also trivial. 
Thus both  $G_0$ and $G_2$  belong to $ \DD_\epsilon$.

 Assume  $G_0\in \DD_\tau$ and $G_2\in \DD_\epsilon$. We need to show $G_1$ belongs to $ \DD_{\epsilon+\tau}$. Let $v$ be in ${\mathbb Q}^r$ such that $D(v,\tau+ \epsilon)\not=\infty$.  Since $D(v,\epsilon)\leq  D(v,\tau+ \epsilon)\geq D(D(v,\epsilon),\tau)$, we have $D(v,\epsilon)\not=\infty\not=D(D(v,\epsilon),\tau)$ and thus we can form the following commutative diagram
 where the indicated homomorphisms are trivial:
	\[\xymatrix@C=20pt@R=25pt{
		0 \rto &  G_0(v) \rto \dto & G_1(v) \rto \dto^{G_1(v\leq D(v,\epsilon))}  & G_2(v) \rto \dto_0 & 0 \\
		0 \rto &  G_0(D(v, \epsilon)) \rto \dto_0 & G_1(D(v, \epsilon)) \rto 
		\dto^{G_1(D(v,\epsilon)\leq D(D(v,\epsilon),\tau))}  & G_2(D(v, \epsilon)) \rto \dto & 0 \\
		0 \rto &  G_0\big(D(D(v, \epsilon), \tau)\big) \rto & G_1\big(D(D(v, \epsilon), \tau)\big) \rto  & G_2\big(D(D(v, \epsilon), \tau)\big) \rto & 0
	}\]
Commutativity and exactness implies that the middle vertical composition:  
\[G_1\big(D(v,\epsilon)\leq D(D(v,\epsilon),\tau)\big)\circ G_1\big(v\leq D(v,\epsilon)\big)=G_1\big(v\leq D(D(v,\epsilon),\tau)\big)\] 
is also the zero homomorphism. Consequently so is the composition:
\[G_1\big(v\leq D(v,\epsilon+\tau)\big)=G\big( D(D(v,\epsilon),\tau)\leq D(v,\epsilon+\tau)\big)\circ G\big(v\leq D(D(v,\epsilon),\tau)\big)\]
This means  $G_1$ belongs to $ \DD_{\epsilon+\tau}$.
\medskip

\noindent
{\bf Simplicity.}
A direct sum of zero homomorphisms is a zero homomorphism. The noise system  $\D$ is therefore closed under direct sums. Let $G$ be a finitely generated tame functor. We are going to show that $G[\tau]\subset G$, defined in~\ref{not svcadfvadfva}, is the minimal subfunctor for which the quotient $G/G[\tau]$ belongs to $\DD_\tau$. Since
by construction $\beta_0G[\tau]\leq \beta_0G$,  we could thus conclude  $\D$ to be simple.

Let  $\{g_s\in G(v_s)\}_{s=1}^{n}$  be  a  minimal set of generators of $G$. By definition 
 $G[\tau]\subset G$ is the subfunctor generated by  $\{h_s\ | \ 1\leq s\leq n\text{ and } D(v_s,\tau)\not=\infty\}$.
First we  show   $G/(G[\tau])$ belongs to $\DD_\tau$. Let $v$ be an object in
${\mathbb Q}^r$  such that $D(v,\tau)\not=\infty$. Any  element   $x$   in  $G(v)$ can be written as $x=\sum_{v_s\leq v} \lambda_s G(v_s\leq v)(g_s)$ and thus:
\[
G\big(v\leq D(v,\tau)\big)(x)=\sum_{v_s\leq v}\lambda_s G\big(D(v_s,\tau)\leq D(v,\tau )\big)\circ G\big(v\leq D(v_s,\tau)\big)(g_s)=\]
\[=\sum_{v_s\leq v}\lambda_s G\big(D(v_s,\tau)\leq D(v,\tau )\big)(h_s)
\]
This implies that $G\big(v\leq D(v,\tau)\big)$ maps  any $x$ in $G(v)$ into $G[\tau]$. Consequently the homomorphism
$G/(G[\tau])\big(v\leq D(v,\tau)\big)$ is trivial.

 Let  $G_0\subset G$ be  a  tame functor for which
$G/G_0$ belongs to  $\DD_\tau$.  Thus, for any $s$ such that $D(v_s,\tau)\not=\infty$,
the homomorphism $G\big(v_s\leq D(v_s,\tau)\big)$ maps the generator  $g_s$ to  an element in  $G_0$, and so  $h_s=G\big(v_s\leq D(v_s,\tau)\big)(v_s)$  belongs to $G_0$. Consequently
the functor $G[\tau]$, generated by these $h_s$, is a subfunctor of $G_0$.
\end{proof}

\begin{example}\label{ex standardpc}
Choose an object $w$ in ${\mathbb Q}^r$. For any  $(v,\epsilon)$    in  ${\mathbb Q}^r_{\infty}\times {\mathbb Q}$,
define:
\[S_w(v,\epsilon):=v+\epsilon w\] 
In particular,  according to  our convention $\infty+\epsilon w=\infty$,
we have $S_w(\infty,\epsilon)=\infty$.  It is straightforward to verify that  $S_{w}$ is a persistence contour.
For this contour the following  equality  holds  $S_u\big(S_w(v,\epsilon),\tau\big)=v+\epsilon w+\tau w= S_w(v,\epsilon+\tau)$.
The functor $S_w$ is called the   {\bf standard} persistence contour in the direction $w$.
 The associated simple noise system is given by:
\[({\mathcal S}_w)_{\epsilon}=\left\{
\begin{array}{l|c}
\multirow{2}{*}{$G\colon{\mathbf Q}^r\to \text{Vect}_K$} &G
\text{ is a finitely generated tame functor such that }\\
 & G(v\leq v+\epsilon w)\text{ is trivial for any $v$}
\end{array}
\right\}
\]
In~\cite{martina} this noise system is called the standard noise in the direction of  $w$.
The $\tau$-shift $G_{{\mathcal S}_w}[\tau]\subset G$  is given by  the subfunctor generated by all the elements of the form
$ G(v_s\leq v_s+\tau w)(g_s)$ where $\{g_s\in G(v_s)\}_{s=1}^n$  is a minimal set of generators of $G$.
For example, consider $v\leq u$    in ${\mathbf Q}^r$. The functor $[v,u)$ (see~\ref{pt bar}) belongs to $({\mathcal S}_w)_{\epsilon}$ if and only if  $u\leq v+\epsilon w$.
\end{example}

We are now ready to state and prove a constructive characterization of simple noise systems. Any such noise is associated with a unique  persistence contour. 

\begin{thm}\label{thm:contour bijection}
The function that assigns to a persistence contour  the noise system defined in~\ref{not svcadfvadfva}  is a bijection between  the set of persistence contours and  the set of simple noise systems.
\end{thm}
\begin{proof}
We are going to prove that the construction~\ref{not adfaergsfhdg} of a persistence contour associated to  a simple noise system 
is the inverse to the function claimed to be a bijection in the theorem.
 
Let  $D\colon {\mathbf Q}^r_{\infty}\times {\mathbf Q}\to {\mathbf Q}^r_{\infty}$ be a persistence contour and 
${\mathcal D}'=\{{\mathcal D}'_{\epsilon}\}_{\epsilon\in {\mathbb Q}}$ be the associated simple noise system
as defined in~\ref{not svcadfvadfva}. 
In~\ref{prop svfwrtg} it was shown that, for a finitely generated tame functor $G$, the subfunctor $G[\tau]\subset G$ constructed in~\ref{not svcadfvadfva} coincides with $G_{{\mathcal D}'}[\tau]$, i.e., it is the smallest  subfunctor for which $G/G[\tau]$ belongs to ${\mathcal D}'_{\tau}$. According to this construction:
\[K(v,-)[\tau]=\begin{cases}
0& \text{if } D(v,\tau)=\infty\\
K\big(D(v,\tau)-\big) &\text{if } D(v,\tau)\not=\infty
\end{cases}\]
Thus the persistence contour, constructed in~\ref{not adfaergsfhdg} for  the noise system  ${\mathcal D}'$,   is equal to $D\colon {\mathbf Q}^r_{\infty}\times {\mathbf Q}\to {\mathbf Q}^r_{\infty}$.

Let ${\mathcal D}=\{{\mathcal D}_{\epsilon}\}_{\epsilon\in {\mathbb Q}}$ be a simple noise system. Consider the associated  persistence contour $D\colon {\mathbf Q}^r_{\infty}\times {\mathbf Q}\to {\mathbf Q}^r_{\infty}$  
as defined in~\ref{not adfaergsfhdg}. Let  ${\mathcal D}'=\{{\mathcal D}'_{\epsilon}\}_{\epsilon\in {\mathbb Q}}$ be the noise system associated to the contour $D$ as define in~\ref{not svcadfvadfva}.
We need to show ${\mathcal D}={\mathcal D}'$. Since both of these noise systems 
are closed under  direct sums,   it is enough  to prove  ${\mathcal D}$ and ${\mathcal D}'$  contain the same cyclic functors (see~\ref{prop cyclcialldirsum}).

Let $G$ be a cyclic functor. Choose  an epimorphism  $\pi\colon K(v,-)\twoheadrightarrow G$. 
 Assume $D(v,\tau)=\infty$. This has two consequences. First,  $K(v,-)_{\mathcal D}[\tau]=0$ and 
 so $K(v,-)$ is a member of ${\mathcal D}_{\epsilon}$.  Second, $K(v,-)$ belongs to ${\mathcal D}'_{\tau}$
 (see the paragraph after~\ref{not svcadfvadfva}).
 As its quotient, $G$ then belongs to both ${\mathcal D}_{\tau}$ and ${\mathcal D}'_{\tau}$.
 
 Assume $D(v,\tau)\not=\infty$. Then $K(v,-)_{{\mathcal D}}[\tau]=K\big(D(v,\tau),-\big)$ and we have a commutative square, where the indicated arrows are epimorphisms (see~\ref{prop:exmin}):
 \[\xymatrix{
K\big(D(v,\tau),-\big)\ar@{^(->}[r]\ar@{->>}[d] & K(v,-)\ar@{->>}[d]\\
G_{\mathcal D}[\tau]\ar@{^(->}[r] & G
}\]
Commutativity of this diagram implies that  the homomorphism $G\big(v\leq D(v,\tau)\big)$ is trivial if and only
the homomomorphism $K\big(D(v,\tau),-\big)\twoheadrightarrow G_{\mathcal D}[\tau]$ is trivial.
As this is an epimorphism we get an equivalence: $G_{\mathcal D}[\tau]=0$ if and only if 
 $G\big(v\leq D(v,\tau)\big)$ is trivial. The statement  $G_{\mathcal D}[\tau]=0$  is equivalent to
 $G$ belonging to ${\mathcal D}_{\tau}$ and the homomorphism $G\big(v\leq D(v,\tau)\big)$ being  trivial is equivalent to
 $G$ belonging to ${\mathcal D}'_{\tau}$. We can  conclude that ${\mathcal D}_{\tau}$ and ${\mathcal D}'_{\tau}$ contain the same cyclic functors. 
 \end{proof}
 
 \begin{point}\label{pt trancations}
{\bf Truncations.} Let  $C\colon  {\mathbf Q}^r_{\infty}\times {\mathbf Q}\to {\mathbf Q}^r_{\infty}$ be a persistence contour  and ${\mathcal C}=\{{\mathcal C}_{\epsilon}\}_{\epsilon\in{\mathbb Q}}$  be the associated simple noise system (see~\ref{not svcadfvadfva}). Choose  an element  $u$  in ${\mathbb Q}^r$.  For $(v,\epsilon)$ in  
${\mathbb Q}^r_{\infty}\times {\mathbb Q}$ define:
\[
(C/u)(v,\epsilon):=\begin{cases}
C(v,\epsilon) & \text{if } u\not\leq C(v,\epsilon)\\
\infty & \text{if } u\leq C(v,\epsilon)
\end{cases}
\]
The fact that   $C/u$ is a functor   and  the relation $v\leq (C/u)(v,\epsilon)$ are clear.  Thus to show  $C/u$ is a  persistence contour, it remains  to prove  
$ (C/u)\big( (C/u)(v,\epsilon),\tau\big)\leq(C/u)(v,\epsilon+\tau)$. This inequality is obvious if 
$u\leq C(v,\epsilon)$ as in this case both sides  are equal to $\infty$. Assume $u\not \leq C(v,\epsilon)$. Since 
$v\leq C(v,\epsilon)$, then also  $u\not\leq v$. Thus in this case $ (C/u)\big( (C/u)(v,\epsilon),\tau\big)=
C\big( C(v,\epsilon),\tau\big)$ and $(C/u)(v,\epsilon+\tau)= C(v,\epsilon+\tau)$ and  the required inequality also holds in this case as $C$ is a persistence contour. The contour $C/u$  is called
the {\bf truncation} of $C$ at $u$. 

For example, let $S_w\colon  {\mathbf Q}^r_{\infty}\times {\mathbf Q}\to {\mathbf Q}^r_{\infty}$ be the standard persistence contour in the direction of $w$ (see~\ref{ex standardpc}). Then
\[
(S_{w}/u)(v,\epsilon):=\begin{cases}
v+\epsilon w & \text{if } u\not\leq v+\epsilon w\\
\infty & \text{if } u\leq v+\epsilon w
\end{cases}
\]

The noise system ${\mathcal C}/u$, associated  with the truncated contour $C/u$, is  called
the {\bf truncation} of  ${\mathcal C}$ at $u$.
This simple  noise system ${\mathcal C}/u$
is the smallest noise system such that $K(u,-)$ belongs to $({\mathcal C}/u)_{0}$  and ${\mathcal C}_{\epsilon}\subset ({\mathcal C}/u)_{\epsilon}$ for any $\epsilon$ in ${\mathbb Q}$.  

Let $G\colon{\mathbf Q}^r\to\text{Vect}_K$ be a finitely generated tame functor and $\{g_s\in G(v_s)\}_{s=1}^n$  be its  a minimal set of generators.
The    translation $G_{{\mathcal C}/u}[\tau]\subset G$ is   the subfunctor generated by the following set:
\[\{G(v_s\leq v_s+\tau w)(g_s)\ |\ \text{for $s$ such that }C(v_s,\tau)\not=\infty\text{ and } u\not\leq C(v_s,\tau)\}\]
\end{point}

\section{The case $r=1$}\label{sec:r1}
Throughout this section $r=1$. Let us also fix a simple noise system ${\mathcal C}=\{{\mathcal C}_{\epsilon}\}_{\epsilon\in{\mathbb Q}}$ in ${\mathbb T}({ \mathbf  Q}, \text{Vect}_K)$. Let $C\colon  {\mathbf Q}_{\infty}\times {\mathbf Q}\to {\mathbf Q}_{\infty}$ be the associated persistence contour. 

For $v\leq w$   in ${\mathbf Q}$,  functors of the form $[v,w)$  (see~\ref{pt bar})  and
$K(v,-)$ are called   {\bf bars} of length $w-v$ and $\infty$ respectively. The bars are the indecomposable elements
in ${\mathbb T}({ \mathbf  Q}, \text{Vect}_K)$, i.e.\  any finitely generated tame functor indexed by ${\mathbf Q}$ 
 is isomorphic to a finite direct sum of bars and  the isomorphism types of  these   summands are uniquely determined by the isomorphism type of the functor (see for example~\cite{MR3575998},~\cite{MR3524869}, \cite[Proposition 5.6]{martina}).    Such a direct sum decomposition   is called a 
 {\bf barcode}. The rank   $\beta_0G$  is equal to  the number  of  bars in a barcode  of $G$.  
 
 A subfunctor of a finitely generated free functor is also free and of smaller or equal rank.  Consequently, for a tame subfunctor $F\subset G$, there is an inequality $\beta_0F\leq \beta_0G$.  This together with~\ref{thm calsimple}  implies:
 
 \begin{prop} 
 For any finitely generated tame functor 
 $G\colon {\mathbf Q}\to  \text{\rm Vect}_K$ and $\tau$   in   ${\mathbb Q}$:
 \[\widehat{\beta}_0 G(\tau)=\beta_0 G[\tau]\]
 \end{prop}
 
 Since taking the shift commutes with direct sums (see~\ref{prop:exmin}.(3)), according to the  above proposition, so does the stabilization of the rank:  if $F,G\colon {\mathbf Q}\to  \text{\rm Vect}_K$  are finitely generated and tame, then for any  $\tau$   in   ${\mathbb Q}$
 \[\widehat{\beta}_0 (F\oplus G)(\tau)=\beta_0 (F\oplus G)[\tau]=
 \beta_0 (F[\tau]\oplus G[\tau])=\]
 \[\beta_0 F[\tau]+ \beta_0 G[\tau]=\widehat{\beta}_0 F(\tau)+
 \widehat{\beta}_0 G(\tau)\]
 Thus if we decompose $G$ as a direct sum of bars $\oplus_{s=1}^n B_s$, then:
 \[\widehat{\beta}_0 G(\tau)=\widehat{\beta}_0(\oplus_{s=1}^n B_s)(\tau)=
 \sum_{s=1}^n \widehat{\beta}_0 B_s(\tau)= \sum_{s=1}^n\beta_0B_s[\tau]
 \]
 Recall:
 \[\beta_0 K(v,-)[\tau]=\begin{cases}
 1 &\text{if } C(v,\tau)\not=\infty \text{ or equivalently if } K(v,-)\not\in{\mathcal C}_\tau\\
 0 &\text{if } C(v,\tau)=\infty \text{ or equivalently if } K(v,-)\in{\mathcal C}_\tau\
 \end{cases}
 \]
 \[\beta_0 [v,w)[\tau]=\begin{cases}
 1 &\text{if } C(v,\tau)<w \text{ or equivalently if } [v,w)\not\in{\mathcal C}_\tau\\
 0 &\text{if } C(v,\tau)\geq w  \text{ or equivalently if } [v,w)\in{\mathcal C}_\tau\
 \end{cases}
 \]
 We can conclude that,  for a finitely generated tame functor $G\colon {\mathbf Q}\to  \text{\rm Vect}_K$,   $\widehat{\beta}_0 G(\tau)$ is equal to the number of bars, in a barcode of $G$, that do not belong to ${\mathcal C}_\tau$ (a more general statement can be found in~\cite[Section 10]{martina}).
 
 \begin{example}\label{ex standtranc}
 Let ${\mathcal S}$ be the simple  noise system in  ${\mathbb T}({ \mathbf  Q}, \text{Vect}_K)$
 associated with the standard contour $S\colon  {\mathbf Q}_{\infty}\times {\mathbf Q}\to {\mathbf Q}_{\infty}$ given by  $S(v,\epsilon)=v+\epsilon$ (see~\ref{ex standardpc}).  For this noise system, the  $\tau$-shift of
  $G=\big(\bigoplus_{\substack{1\leq s\leq n}} [v_s,w_s)\big)\oplus \big( \bigoplus_{\substack{1\leq t\leq m}} K(u_t,-)\big)$ has the following barcode:
  \[G_{\mathcal S}[\tau]=
 \bigoplus_{\substack{1\leq s\leq n\\ \tau<w_s-v_s}} [v_s+\tau,w_s)\oplus
  \bigoplus_{\substack{1\leq t\leq m}} K(u_t+\tau,-) \] 
  The value $\widehat{\beta}_0 G(\tau)=\beta_0G_{\mathcal S}[\tau]$ is therefore  given by the number of bars, in a barcode of $G$, of length strictly bigger  than $\tau$.
  Thus the invariant  $G \mapsto \widehat{\beta}_0 G$  encodes information about the length of bars in a barcode of $G$ but not  where  the bars  start and end.  This loss of information can be recovered however by  considering truncations of the standard noise (see~\ref{pt trancations}).
  Choose $u$    in $\mathbb Q$ and  consider  the truncated noise system ${\mathcal S}/u$ 
(see~\ref{pt trancations}). Recall
that ${\mathcal S}/u$ is associated with the contour given by:
\[ (S/u)(v,\epsilon)= \begin{cases}
v+\epsilon&\text{if } v+\epsilon< u\\
\infty &\text{if } u\leq v+\epsilon
\end{cases}\]
For this truncated noise system the $\tau$-shift
of the functor $G$ has the following barcode:
\[G_{{\mathcal S}/u}[\tau]=
\bigoplus_{\substack{1\leq s\leq n\\ \tau<w_s-v_s\\ v_s<u-\tau}} [v_s+\tau,w_s)\oplus
  \bigoplus_{\substack{1\leq t\leq m\\ u_t<u-\tau}} K(u_t+\tau,-) 
\]
The value $\widehat{\beta}_0 G(\tau)=\beta_0G_{{\mathcal S}/u}[\tau]$ is therefore  given by the number of bars, in a barcode of $G$, of length strictly bigger  than $\tau$ and whose starting
points are strictly smaller than $u-\tau$.
 \end{example}
 
   \begin{prop}\label{prop recovstandard}
  Let $\mathcal S$ be the standard noise system considered in~\ref{ex standtranc}. Then  two finitely generated tame functors  $F,G\colon {\mathbf Q}\to  \text{\rm Vect}_K$ are isomorphic if and only if  $\beta_0F_{{\mathcal S}/u}[\tau]=\beta_0G_{{\mathcal S}/u}[\tau]$
  for any $\tau$ and $u$ in ${\mathbb Q}$.
  \end{prop}
  \begin{proof} 
The if implication is obvious. Assume $\beta_0F_{{\mathcal S}/u}[\tau]=\beta_0G_{{\mathcal S}/u}[\tau]$ for any  $\tau$ and $u$ in ${\mathbb Q}$.  Fix barcodes of $F$ and $G$.  If we choose  $u$  strictly bigger than  the beginnings of the  bars in the decompositions of $F$ and $G$, then
the equalities $\beta_0F=\beta_0F_{{\mathcal S}/u}[0]=\beta_0G_{{\mathcal S}/u}[0]=\beta_0G$
imply  then decompositions of $F$ and $G$ have the same number of bars.
To show  $F$ and $G$ are isomorphic we proceed by induction on the number of different lengths of bars in the decomposition of $F$.

Assume first  all the bars in the decomposition of $F$ have the same length $l_F$. If there is a bar in $G$ of length $l$ different than $l_F$, then the numbers $\beta_0F_{{\mathcal S}/u}[\tau]$  and  $\beta_0G_{{\mathcal S}/u}[\tau]$ would be  different if we choose  $\text{max}\{l_F,l\}>\tau> \text{min}\{l_F,l\}$ and 
 $u$ to be  strictly larger than  $\tau$ plus the beginnings of all the  bars in the decompositions of $F$ and $G$.
 Thus all the bars in the decomposition of $G$ have also the same length $l_F$. By setting $\tau=0$ and varying $u$, we can then see  that the number of bars in $F$ and $G$,  which start  at a particular point, are the same.
Consequently $F$ and $G$ are isomorphic.

Assume  $F$ has bars of different length in its barcode.
Let $l_{F}$  and $l_G$   be the biggest  length and $l'_{F}$ and $l'_{G}$ be the next biggest length of any bar in the decomposition of $F$ and  $G$ respectively. 
Choose    $\tau$ satisfying $l_{F}>\tau>l'_{F}$ and $u$ to be bigger than $\tau$ plus the beginnings of all the bars in the decomposition of $F$.   Since 
$\beta_0G_{{\mathcal S}/u}[\tau]=\beta_0F_{{\mathcal S}/u}[\tau]\not=0$, then 
$G$ contains bars of length bigger than $\tau$. As this holds for any $\tau$ satisfying $l_{F}>\tau>l'_{F}$,
 we can conclude  $l_{G}\geq l_{F}$.  If  $l_{G}> l_{F}$, then for $l_{G}> \tau>l_{F}$ and $u$ bigger than $\tau$ plus the   beginnings of all the bars in the decomposition of $G$, the numbers 
 $\beta_0F_{{\mathcal S}/u}[\tau]$  and  $\beta_0G_{{\mathcal S}/u}[\tau]$ 
 would be  different.  We therefore have $l_{G}= l_{F}$. Furthermore the number of bars of length
 $l_{F}$ in the  decompositions of both $F$ and $G$   is the same. Let $F'$ and $G'$ be the direct sums of all the bars of length $l_F$ in the decompositions of respectively  $F$ and $G$.  Note that for  $\tau\geq \text{max}\{l'_F,l'_G\}$  and any $u$:
 \[\beta_0F'_{{\mathcal S}/u}[\tau]=\beta_0F_{{\mathcal S}/u}[\tau]=
 \beta_0G_{{\mathcal S}/u}[\tau]=\beta_0G'_{{\mathcal S}/u}[\tau]
 \]
 Furthermore for any $\tau'\leq \tau < l_F$  and any $u$:
\[ \beta_0F'_{{\mathcal S}/u}[\tau']=\beta_0F'_{{\mathcal S}/u}[\tau]\ \ \ \ \ \ \ \  
\beta_0G'_{{\mathcal S}/u}[\tau']=\beta_0G'_{{\mathcal S}/u}[\tau]\]
 It follows  $\beta_0F'_{{\mathcal S}/u}[\tau]=\beta_0G'_{{\mathcal S}/u}[\tau]$  for any $\tau$ and $u$.
 According to  the first step of the induction, the functors $F'$ and $G'$ are therefore isomorphic. 
  Let $F''$ and   $G''$ be the direct sums of all the bars of length different than $l_F$ in the decompositions of   $F$ and $G$ respectively.   Then $F$ is isomorphic to $F'\oplus F''$ and $G$ is isomorphic to $G'\oplus G''$. 
  By additivity of the shift and the rank,  for any $\tau$ and $u$:
   \[\beta_0F''_{{\mathcal S}/u}[\tau]=\beta_0F_{{\mathcal S}/u}[\tau]-
   \beta_0F'_{{\mathcal S}/u}[\tau]=\]
\[=   \beta_0G_{{\mathcal S}/u}[\tau]-
   \beta_0G'_{{\mathcal S}/u}[\tau]=
   \beta_0G''_{{\mathcal S}/u}[\tau]\]
By the  inductive hypothesis   $F''$ and $G''$ are  isomorphic and so are   $F$ and $G$.  
  \end{proof}

 \section{The case $r\geq 2$.}\label{sec NPhardness}
 Let $K$ be a finite  field.
 In~\ref{cor redfinite} we proved that, for a  simple noise system,    calculating  $\widehat{\beta}_0G(\tau)$ requires  only finitely many operations.  For $r\geq 2$,  one should not  however expect to find an efficient algorithm to do that. The aim of this section is to show   that  calculating $\widehat{\beta}_0G(\tau)$ is in general an  NP-hard problem if $r\geq 2$. This is in contrast with the case of $r=1$ where
 the complexity is closely related to the complexity of Gaussian elimination (this is the complexity of finding barcodes of such functors, see for example~\cite{MR2121296},~\cite{Edelsbrunner2002}).
 
 To show this NP-hardness we are going to compare calculating  $\widehat{\beta}_0G(\tau)$  with   the RANK-3 problem introduced  in \cite{minrank}.  The RANK-3 problem is an example of a more general problem which we now describe.
 
 \begin{point}
 {\bf Input:}    a sequence $x_1,\ldots, x_k$ of vectors in $K^n$
  and a sequence $L_1,\ldots, L_k$ of vector subspaces of $K^n$ each given by  a system of linear equations.
\end{point}

 \begin{point}{\bf Output:}\label{pt output} \medskip$\text{min}\{\text{dim}(L)\  |\ L\text{ is a subspace of $K^n$ s.t.\  }
x_s\in L+L_s\text{ for  $1\leq s\leq k$}\}.$
\end{point}

 
 The above problem can be reformulated in terms of  finding the  smallest rank of a matrix in a certain  set of matrices.
An $n\times k$ matrix $A=\begin{bmatrix}c_1 &\cdots & c_k \end{bmatrix}$  is said to {\bf belong} to the input above if, for any $1\leq s\leq k$,  the vector $c_s-x_s$ is in   $ L_s$.   Note that if $A$ belongs to the input, then  
$x_s\in \text{span}(c_1,\ldots,c_k)+L_s$ for any $s$.  On the other hand if, for all $s$,  $x_s=c_s+y_s$ where  $c_s$ is in 
$L$ and $y_s$ is in $L_s$, then the matrix $\begin{bmatrix}c_1 &\cdots & c_k \end{bmatrix}$ belongs to the input.
It is then  clear that the number in the output above coincides with:
\[\text{min}\{\text{rank(A)}\ |\ A \text{ belongs to the input}\}\]

The following two sources of inputs  are of primary interest to us:
\begin{example}\label{ex graph}
Let   $X=(V,E)$ be a finite graph with $V=\{1,2,\ldots, n\}$.  
Consider the following collection of $n\times n$ matrices $A$ with coefficients in $K$:
\[\mathcal{A}(X):=\{A\ |\ A_{ss}=1 \text{ for any } s\in V\text{ and } A_{st}=0\text{ for any } \{s,t\}\in E\}\]
For example if  $X$ is colored by  $\chi(X)$ (the chromatic number of $X$)  colors, then the  $n\times n$ matrix
$M$, for which   $M_{st}=1$ if $s$ and $t$ have the same color and
 $M_{st}=0$ otherwise, belongs to  the collection $\mathcal{A}(X)$.  Since   this matrix is of rank $\chi(X)$
 (see~\cite[Section 1]{minrank}),
 the collection $\mathcal{A}(X)$ satisfies   the assumptions of~\cite[Theorem 3.1]{minrank}.
According to this  theorem, for an arbitrary graph $X$,  the RANK-3 problem of  deciding if $\mathcal{A}(X)$  contains a matrix of rank $3$ is  NP-complete.
 
Note that  $\mathcal{A}(X)$ consists of all the matrices that 
  belong to the input   given by the sequence $e_1,\ldots, e_n$   of  vectors in the standard basis for $K^n$ and the sequence $L_1,\ldots, L_n$ of vector subspaces of $K^n$
where $L_s$ is defined as
$L_s=\{y\in  K^n\ |\  y_t=0 \text{ if }  \{s,t\}\in E \text{ or } t=s\}$.
We can therefore  conclude that deciding:
 \[\text{min}\{\text{rank(A)}\ |\ A \text{ belongs to the input}\}=\text{min}\{\text{rank(A)}\ |\ A\in\mathcal{A}(X) \}\leq 3\]
 is an NP-complete problem.   Thus in general the problem of calculating~\ref{pt output} is as hard as the RANK-3 problem.
\end{example}

	


\begin{example}\label{ex gennphard}
Let $C\colon  {\mathbf Q}^r_{\infty}\times {\mathbf Q}\to {\mathbf Q}^r_{\infty}$ be a  persistence contour and 
$\mathcal C$   the associated noise system.  {\bf The initial data} consists of:  
\begin{itemize}
\item  a tame functor  $G\colon{\mathbf Q}^r\to\text{Vect}_K$, 
\item  its minimal set of generators
  $\{g_s\in G(v_s)\}_{s=1}^{n}$, 
  \item  an element $\tau$  in ${\mathbb Q}$, 
  \item   an  element $u$   in ${\mathbb Q}^r$ such that $v_s\leq u\leq C(v_s,\tau)$ for any $1\leq s\leq n$. 
  \end{itemize} With this initial  data we can form the following {\bf input}:
\begin{itemize}
\item  the vector space $G(u)$,
\item  the sequence   
$G(v_1\leq u)(g_1),\ldots, G(v_n\leq u)(g_n)$ of vectors  in $ G(u)$, 
\item the sequence $L_1,\ldots L_n$  of vector subspaces  of  $G(u)$ defined as follows:
\[L_s=\begin{cases}
 \text{Ker}\ G(u\leq C(v_s,\tau)) &\text{if } C(v_s,\tau)\not=\infty\\
 G(u) &\text{if } C(v_s,\tau)=\infty
\end{cases}\]
\end{itemize}
We claim that the output for the input above coincides with $\widehat{\beta}_0G(\tau)$:
\smallskip

\noindent
{$\text{\bf Output}\geq \widehat{\beta}_0G(\tau)$:}\quad
  Let $L$ be the subspace
of $G(u)$ of the smallest dimension such that $G(v_s\leq u)(g_s)\in L+L_s$ for $1\leq s\leq n$. Choose a basis $\{f_1,\ldots, f_m\}$ for $L$ and consider  the subfunctor  $F\subset G$ generated by the elements $\{f_s\in G(u)\}_{s=1}^m$. Note that $\text{rank}(F)=m$. Moreover, since the $\tau$-shift  $G[\tau]$ is generated by the elements $\{G(v_s\leq C(v_s,\tau))(g_s)\ |\ 1\leq s\leq n\text{ and } C(v_s,\tau)\not=\infty\}$,  the functor $F$ contains  $G[\tau]$. Thus according to Theorem~\ref{thm calsimple}   $m\geq \widehat{\beta}_0G(\tau)$. 
\smallskip

\noindent
{$\text{\bf Output}\leq \widehat{\beta}_0G(\tau)$:}\quad
Let $H$ be a finitely generated subfunctor of $G$ of the smallest rank such that $G[\tau]\subset H\subset G$.
Choose a minimal set of generators $\{h_i\in H(w_i)\}_{i=1}^m$. Since $H$ is a subfunctor of $G$ we can express its generators as linear combinations:
\[h_i=\sum_{v_s\leq w_i}a_{is}G(v_s\leq w_i)(g_s)\]
We use the  coefficients $a_{is}$ to define the following vectors in  $G(u)$ for $1\leq i\leq m$:
\[f_i=\sum_{v_s\leq w_i}a_{is}G(v_s\leq u)(g_s)\]
Define $L$ to be the subspace of $G(u)$ generated by these $f_1,\cdots, f_m$. 
Choose $s$ such that  $1\leq s\leq n$  and   $C(v_s,\tau)\not=\infty$. The fact that  $H$ contains the $\tau$-shift $G[\tau]$ implies that the vector
 $G(v_s\leq C(v_s,\tau))(g_s)$ can be expressed as a linear
combination of  $\{G(w_i\leq  C(v_s,\tau))(h_i)\ |\ w_i\leq v_s\}$. 
Since $G$ is a functor, the element $G(v_s\leq C(v_s,\tau))(g_s)$ can then    also be the expressed as a linear
combination of $\{G(u\leq  C(v_s,\tau))(f_i)\}_{i=1}^m$.
This means  that the vector 
$G(v_s\leq u)(g_s)$ belongs to   $L+ \text{Ker}\ G(u\leq  C(v_s,\tau))$.
Since  $\text{dim}(L)\leq m$ the claimed inequality follows. 
\end{example}

It turns out that Example~\ref{ex graph} is a special case of~\ref{ex gennphard} and our next step is to explain 
how the input given by a graph in~\ref{ex graph}  can be described as an input induced by a tame functor  as explained in~\ref{ex gennphard}. We do that with a help of so called band functors:

\begin{point}{\bf Band Functors.}
Let $n\geq 0$ be a natural number and $L_0,\ldots, L_n$ be a sequence of subspaces of $K^{n+1}$.
A band functor $B(L_0,\ldots, L_n)\colon {\mathbf Q}^2\to \text{Vect}_K$ is by definition a $1$-tame functor whose $1$-frame
(see~\ref{pt tamness})
is constructed as follows:
\smallskip

\noindent
{\bf $1$-Frame:}
Let $P\colon {\mathbf N}^2\to \text{Vect}_K$ be the free functor given by the direct sum   $P:=\oplus_{s=0}^{n} K((n-s,s),-)$. 
Note that  $(2n-t,n+t)\geq (n-s,s)$  for   any $0\leq s, t\leq n$. Thus  $P(2n-t,n+t)=\oplus_{s=0}^{n} K((n-s,s),(2n-t,n+t))=K^{n+1}$.
For $0\leq  t\leq n$, we regard  $L_t$ to be the subspace of $K^{n+1}=P(2n-t,n+t)$. Define  the frame  $B\colon {\mathbf N}^2\to\text{Vect}_K$ to be the unique functor for which the following homomorphisms
form a surjective   natural transformation  $\pi\colon P\to B$:
\[\pi_{(a,b)}\colon P(a,b)\to B(a,b)\text{ is given by }\]
\[\begin{cases}
\text{id} & \text{ if }a\leq 2n, b\leq 2n, a+b< 3n\\
\text{the quotient } K^{n+1}\to K^{n+1}/L_{b-n}& \text{ if } a\leq 2n, n\leq b\leq 2n, a+b= 3n\\
P(a,b)\to 0 &\text{ otherwise }
\end{cases}
\]

The following tables give the non-zero values of $B$  for $n=1, 2,$ and $3$:
\[\begin{array}{c|c|c|c|c}
\scriptstyle{3} & 0 & 0 & 0 & 0\\ \hline
\scriptstyle{2} & K & K^2/L_1 & 0 & 0\\ \hline
\scriptstyle{1} & K & K^2 & K^2/L_0 & 0\\ \hline
\scriptstyle{0} & 0 & K & K & 0\\ \hline  
 &\scriptstyle{0} & \scriptstyle{1} & \scriptstyle{2} & \scriptstyle{3}
\end{array}
\ \ \ \ 
\begin{array}{c|c|c|c|c|c|c}
\scriptstyle{5} & 0 & 0 & 0 & 0 & 0 & 0\\ \hline
\scriptstyle{4} & K & K^2& K^3/L_2 & 0 & 0 & 0\\ \hline
\scriptstyle{3} & K & K^2 & K^3 & K^3/L_1 & 0 & 0\\ \hline
\scriptstyle{2} & K & K^2 & K^3 & K^3 & K^3/L_0 & 0\\ \hline
\scriptstyle{1} & 0 & K & K^2 & K^2 & K^2 & 0\\ \hline
\scriptstyle{0} & 0 & 0 & K & K & K  & 0\\ \hline
 & \scriptstyle{0} & \scriptstyle{1} & \scriptstyle{2} & \scriptstyle{3} & \scriptstyle{4} & \scriptstyle{5}
\end{array}
\]
\[\begin{array}{c|c|c|c|c|c|c|c|c}
\scriptstyle{7} & 0 & 0 & 0 & 0 & 0 & 0 & 0 & 0\\ \hline
\scriptstyle{6} & K & K^2 & K^3 & K^4/L_3 & 0 & 0 & 0 & 0\\ \hline
\scriptstyle{5} & K & K^2 & K^3 & K^4 & K^4/L_2 & 0 & 0 & 0 \\ \hline
\scriptstyle{4} & K & K^2& K^3& K^4 & K^4 & K^4/L_1 & 0 & 0\\ \hline
\scriptstyle{3} & K & K^2 & K^3 &K^4 & K^4 & K^4 & K^4/L_0 & 0\\ \hline
\scriptstyle{2} & 0 & K & K^2 & K^3 & K^3 & K^3 & K^3 &  0\\ \hline
\scriptstyle{1} & 0 & 0 & K & K^2 & K^2 & K^2 & K^2 & 0\\ \hline
\scriptstyle{0} & 0 & 0 & 0 & K & K  & K & K & 0\\ \hline
 & \scriptstyle{0} & \scriptstyle{1} & \scriptstyle{2} & \scriptstyle{3} & \scriptstyle{4} & \scriptstyle{5} &  \scriptstyle{6}
 &  \scriptstyle{7}
\end{array}
\]
\smallskip

Fix the standard persistence contour $S\colon  {\mathbf Q}^2_{\infty}\times {\mathbf Q}\to {\mathbf Q}^2_{\infty}$
given by $S(v,\tau)=v+\tau(1,1)$ (see~\ref{ex standardpc}). Consider the following initial data (see~\ref{ex gennphard}):
\begin{itemize}
\item  the band functor  $B(L_0,\ldots, L_n)\colon {\mathbf N}^2\to\text{Vect}_K$, 
\item its minimal set of generators given by
$\{1\in B(n-s,s)=K\}_{s=0}^n$, 
\item the element $\tau=n$ in ${\mathbb Q}$, 
\item the element  $u=(n,n)$ in   ${\mathbb Q}^2$.
\end{itemize}
Note that $(n-s,s)\leq (n,n)\leq (n-s,s)+n(1,1)$  for any $0\leq s\leq n$. Thus this initial data satisfies the 
required assumption (see~\ref{ex gennphard}). We can therefore  conclude that $\widehat{\beta}_0B(\tau)$ coincide with
the output associated with the input given by the sequence of vectors $e_0,\ldots, e_n$ in $K^{n+1}$ and the sequence of subspaces $L_0,\ldots, L_n$ of $K^{n+1}$. Since  no condition was put on these  subspaces we could chose them to coincide with the subspaces  associated to an arbitrary graph as explained  
in Example~\ref{ex graph}. This shows that computing $\widehat{\beta}_0B(\tau)$ is at least as hard as deciding the RANK-3 problem. We can therefore  conclude:
\end{point}

\begin{thm}\label{thm:np}
Let $S\colon  {\mathbf Q}^2_{\infty}\times {\mathbf Q}\to {\mathbf Q}^2_{\infty}$ be the standard persistence contour given by $S(v,\tau)=v+\tau(1,1)$ and $G\colon {\mathbf Q}^2\to\text{\rm Vect}_K$ be an arbitrary finitely generated tame functor.
Deciding $\widehat{\beta}_0G(\tau)\leq 3$   is an NP-complete problem. Consequently 
computing $\widehat{\beta}_0G(\tau)$  is NP-hard. 
\end{thm}
 Using the above result we are now ready to prove our main theorem. 
 
\begin{T1}
Computing the stabilization of the 0-th Betti number is NP-hard.
\end{T1}
\begin{proof}
Note that any $2$-parameter persistence module can be viewed as an $r$-parameter persistence module for $r>2$ where all the maps in the added dimensions are identities. Hence the above statement contains Theorem \ref{thm:np} as a special case, by which the result follows.
\end{proof}


\section{The interleaving distance}
In this section we fix the noise system to be the standard noise system, $\{\mathcal S_\epsilon\}_{\epsilon\in \mathbb Q}$ and let $F, G\in \Tame$. Consider the interleaving distance, as defined in \cite{MR3348168}:
\begin{Def}\label{def:interleaving}
We say that $F$ and $G$ are \textbf{$\epsilon$-interleaved} if there are maps $\varphi_v\colon F(v)\to G(v+\epsilon)$ and $\psi_v\colon G(v)\to F(v+\epsilon)$ such that $\psi_{v+\epsilon}\circ \varphi_v = F(v\leq v+2\epsilon)$ and $\varphi_{v+\epsilon}\circ \psi_v = G(v\leq v+2\epsilon)$. Then
$$d_I(F, G) := \inf \{\epsilon\in [0, \infty)\mid F\text{ and $G$ are $\epsilon$-interleaved}\}$$
\end{Def}
We will show that computing the stabilization of the 0-th Betti number is also NP-hard when using the pseudometric $d_I$. The following proposition show how the two distances relate to each other:
\begin{prop}\label{prop:metriceq}
$$\frac{1}{6}d(F, G) \leq d_I(F, G) \leq d(F, G)$$
\end{prop}
\begin{proof}
%

Suppose that $d_I(F, G)=\epsilon$. Let $H\in \Tame$ be defined by $H(v) := G(v+\epsilon )$ with maps coming from $G$. Then from the map $G(-\leq -+\epsilon)\colon G\to H$ we conclude that $d(G, H) \leq 2\epsilon$. Now since $F$ and $G$ are $\epsilon$-interleaved, there are maps $\varphi_v\colon F(v)\to H(v)$ and $\psi_{v+\epsilon}\colon H(v)\to F(v+2\epsilon)$. By the requirement of commutativity in Defintion \ref{def:interleaving}, we have that the $\varphi_v$ extend to a naural transformation $\varphi\colon F\to H$. Moreover, $\ker{\varphi_v} \subseteq \ker{F(v\leq v+2\epsilon)}$ and thus $\ker{\varphi} \in \mathcal{S}_{2\epsilon}$. To see that $\text{coker}\, \varphi \in \mathcal{S}_{2\epsilon}$, note that there is a section $\psi_{v+\epsilon}$ such that the following diagram commutes:
 \[\xymatrix{
F(v) \ar[r] \ar[d] & H(v) \ar[rd]\ar[d] \ar[ld]_{\psi_{v+\epsilon}} &\\
F(v+2\epsilon) \ar[r] & H(v+2\epsilon)\ar@{->>}[r] & H(v+2\epsilon)/F(v+2\epsilon)
}\]
We conclude that $d(F, H)\leq 4\epsilon$ and thus $d(F, G)\leq 6\epsilon = 6d_I(F, G)$.


Conversely, if $d(F, G)=\epsilon$, there is an $H\in\Tame$ with maps $\varphi\colon H\to F$, $\psi\colon H\to G$ such that $\varphi$ is a $\tau$-equivalence, $\psi$ is a $\mu$-equivalence and $\tau+\mu\leq \epsilon$. We will show that if $\varphi\colon H\to F$ is a $\tau$-equivalence, then $d_I(H, F)\leq \tau$. Suppose that $\ker\varphi\in \mathcal S_a$ and $\text{coker}\varphi\in \mathcal S_b$ such that $a+b\leq \tau$. Then, since $\ker\varphi(v)\to \ker\varphi(v+a)$ is the zero map, there is a section $H/\ker\varphi(v)\to H(v+a)$ such that the following diagram commutes:
 \[\xymatrix{
H(v) \ar[r] \ar@{->>}[d] & H(v+a) \ar@{->>}[d]  \\
H/\ker\varphi(v) \ar@{-->}[ru] \ar[r] & H/\ker\varphi(v+a)
}\]
This shows that $d_I(H, H/\ker\varphi)\leq a$. With an analogous argument we can show that $d_I(H/\ker\varphi, F)\leq b$. By the triangle inequality, this implies that $d_I(H, F)\leq a+b\leq \tau$. Analogously, we find that $d_I(H, G)\leq \mu$ and, again, by the triangle inequality we conclude that $d_I(F, G)\leq \tau+\mu\leq d(F, G)$.
\end{proof}


\begin{remark}\label{rem:subfunctors}
In the above proof, note that if $\varphi\colon G\hookrightarrow F$ is a monomorphism, then $d_I(F, G)=d(F, G)$.
\end{remark}

We now define the stabilization of the 0-th Betti number using the pseudometric $d_I$. For any $\tau$ in $\Bbb Q$:
$$\widehat{\beta}_0^I F(\tau) := \min \{\beta_0 \, G\mid d_I(F, G)\leq \tau\}$$

\begin{thm}
$$\widehat{\beta}_0^I F(\tau) = \widehat{\beta}_0 F(\tau)$$
\end{thm}
\begin{proof}
By Proposition 12.2 we have that $\widehat{\beta}_0^I F(\tau)\geq \widehat{\beta}_0 F(\tau)$. However, by Corollary 8.5 it suffices to optimize over subfunctors, and by Remark \ref{rem:subfunctors} we have that $d_I(F, G)=d(F, G)$ for any subfunctor $G\subseteq F$. Consequently, $\widehat{\beta}_0^I F(\tau)\leq \widehat{\beta}_0 F(\tau)$ and we conclude that the statement holds.
\end{proof}

\begin{T2}
Computing the stabilization of the 0-th Betti number using the interleaving distance is NP-hard.
\end{T2}

\bibliographystyle{plain} 
\bibliography{article}

\end{document}